\documentclass[english]{amsart}

\usepackage{ulem}
\usepackage{graphicx}
\usepackage{amsmath,amsfonts,amsthm}
\usepackage{amssymb}
\usepackage{multirow,array}
\usepackage{comment}
\usepackage{enumitem}
\usepackage{indentfirst}
\usepackage{emptypage}
\usepackage{hyperref}
\usepackage{color}
\usepackage{mathrsfs}

\usepackage[abbrev,msc-links,alphabetic]{amsrefs}

\usepackage{url}

\usepackage[T1]{fontenc}
\usepackage{marvosym}
\usepackage{babel}
\usepackage{amsmath,amsfonts,amsthm}%
\usepackage{amsmath}%
\usepackage{amsfonts}
\usepackage{amssymb}%
\usepackage{graphicx}

\usepackage{shuffle}
\usepackage{tikz}
\usepackage{mhequ}

\usepackage[vscale=0.75]{geometry}
\usetikzlibrary{arrows,snakes,backgrounds}

\usetikzlibrary{snakes}
\usetikzlibrary{decorations}
\usetikzlibrary{positioning}
\usetikzlibrary{shapes}

\newcommand{\ome}{\omega}

\newcommand{\vep}{\varepsilon}

\newcommand{\be}{\begin{equation}}
\newcommand{\ee}{\end{equation}}
\newcommand{\bea}{\begin{eqnarray}}
\newcommand{\eea}{\end{eqnarray}}

\newcommand{\beaa}{\begin{eqnarray*}}
\newcommand{\eeaa}{\end{eqnarray*}}

\newcommand{\gam}{\gamma}

\newcommand{\MH}{\mathcal{H}}
\newcommand{\MF}{\mathcal{F}}

\newcommand{\MC}{\mathcal{C}}
\newcommand{\ML}{\mathcal{L}}

\newcommand{\MZ}{\mathcal{Z}}

\newcommand{\MA}{\mathcal{A}}
\newcommand{\ME}{\mathcal{E}}
\newcommand{\MS}{\mathcal{S}}
\newcommand{\MM}{\mathcal{M}}
\newcommand{\MU}{\mathcal{U}}

\newcommand{\BM}{\mathbb{M}}

\newcommand{\PP}{\mathbb{P}}

\newcommand{\bv}{\bar{v}}

\newcommand{\ty}{\tilde{y}}
\newcommand{\tz}{\tilde{z}}

\newcommand{\tF}{\tilde{F}}

\newcommand{\tTheta}{\tilde{\Theta}}

\newcommand{\by}{\bar{y}}
\newcommand{\bz}{\bar{z}}

\newcommand{\bu}{\bar{u}}

\newtheorem{thm}{Theorem}[section]
\newtheorem{assumption}[thm]{Assumption}
\newtheorem{lem}[thm]{Lemma}
\newtheorem{coro}[thm]{Corollary}
\newtheorem{rem}[thm]{Remark}
\newtheorem{prop}[thm]{Proposition}
\newtheorem{example}[thm]{Example}
\newtheorem{defi}[thm]{Definition}

\usepackage{lineno}
 
\usepackage[textsize=tiny]{todonotes}

\begin{document}

\title[mf equilibria via mf-qBSDEs]{Mean-field quadratic BSDEs and related mean-field portfolio games of controls}

\author{Huilin Zhang}
\address{Research Center for Mathematics and Interdisciplinary Sciences, Shandong University, Binhai Road 72, 266237, Qingdao, PRC;Department of Mathematics, Humboldt University Berlin, Berlin, Germany.}
\email{huilinzhang@sdu.edu.cn}

\begin{abstract}
We study a new class of mean-field quadratic backward stochastic differential
equations (qBSDEs) arising from mean-field portfolio games with exponential utility. Typical examples of such games include a mean-field portfolio game with price impact, and a finite-contract pricing model 
with market clearing conditions. Generators of these mean-field qBSDEs contain quadratic terms $\mathbb{E}[Z]^{\top} Z$ and $|\mathbb{E}[Z]|^2$,
instead of the classical pathwise $Z^\top Z$ term. 
We
prove local well-posedness under $L^q$-integrability assumptions on terminals and their Malliavin derivatives and global well-posedness under an extra exponential integrability condition on the Malliavin
derivatives. Then we  show the existence and uniqueness of global equilibria of the above two mean-field games via our qBSDE theory.

\end{abstract}

\subjclass{93E20, 91A16, 60H30 }
\thanks {\textit{JEL classification.} C73, G11.}

\keywords{quadratic backward stochastic differential equations, mean-field games, Malliavin calculus, price impact, market clearing}

\maketitle

\tableofcontents

\setlength{\abovedisplayskip}{1.5pt plus 1pt minus 0.5pt}
\setlength{\belowdisplayskip}{1.5pt plus 1pt minus 0.5pt}
\setlength{\abovedisplayshortskip}{0pt plus 0.75pt}
\setlength{\belowdisplayshortskip}{1pt plus 0.75pt minus 0.5pt}
\setlength{\jot}{1pt}
\setlength{\arraycolsep}{3.5pt}

%
%

%

\newcommand{\E}{\mathbb{E}}
\newcommand{\tE}{\tilde{\mathbb{E}}}
\newcommand{\R}{\mathbb{R}}
\newcommand{\N}{\mathbb{N}}
\newcommand{\Oc}{ \Omega^c}
\newcommand{\Od}{ \Omega^D}

\newcommand{\D}{\mathbb{D}}
\newcommand{\X}{\mathbb{X}}
\newcommand{\T}{\mathbb{T}}
\newcommand{\bS}{\mathbb{S}}

\newcommand{\BX}{\mathbf{X}}
\newcommand{\bP}{\mathbb{P}}
\newcommand{\bL}{\mathbb{L}}

\newcommand{\tK}{\tilde{{K}}}

\newcommand{\FC}{\mathscr{C}}
\newcommand{\FF}{\mathscr{F}}
\newcommand{\B}{\mathbb{B}}
\newcommand{\BB}{\mathbf{B}}
\newcommand{\op}{\mathcal{P}}
\newcommand{\FD}{\mathscr{D}}
\newcommand{\oq}{\mathcal{Q}}
\newcommand{\oor}{\mathcal {R}}

\newcommand{\BI}{\mathbf{1}}
\newcommand{\BZ}{\mathbf{Z}}
\newcommand{\I}{\mathcal{I}}
 \newcommand{\gl}{{\gamma_t, \eta_t}}

\newcommand{\Z}{\mathbb{Z}}
\newcommand{\M}{\mathbb{M}}
\newcommand{\bx}{\mathbf{x}}
\newcommand{\s}{\mathbb{S}}
\newcommand{\A}{\mathbb{A}}
\newcommand{\LL}{\mathbb{L}}
\newcommand{\tiop}{\tilde{\mathcal{P}}}
\newcommand{\BBI}{\mathbf{I}}
\newcommand{\BY}{\mathbf{Y}}
\newcommand{\FKF}{\mathfrak{K}_F}

\newcommand{\Bf}{\mathbf{f}}
\newcommand{\TSM}{[0,T] \times \D \times \op^D_2}

\newcommand{\Hb}{\MH^2_\text{BMO}}
\newcommand{\BMO}{ \text{BMO}}
\newcommand{\lef}{\left}
\newcommand{\rig}{\right}

\newcommand{\yone}{y^{(1)}}
\newcommand{\zone}{z^{(1)}}
\newcommand{\ytwo}{y^{(2)}}


\section{Introduction}

\subsection{A sketch of mean-field games and models of interest}

Mean-field games (MFGs) are introduced independently by Huang, Malham\'e and
Caines \cite{HuangMalhameCaines06} and by Lasry and Lions
\cite{LasryLions07} to describe Nash equilibria in large populations of
weakly interacting players.  The central reduction is to replace the
multiple-player game by a representative-player control problem with a 
measure flow fixed first and then impose a consistency condition between
that flow and the law generated by the optimal system. In the standard
formulation, players interact primarily through the distribution of their
states.  There are two complementary approaches to this fixed-point problem.
The first is the PDE approach, where one typically works with Markovian controls and characterizes the equilibrium through a second-order PDE on the Wasserstein space, known as the {\it master equation}; see, among many others, Cardaliaguet et al. \cite{cardaliaguet2019master} and Gangbo et al. \cite{GangboMeszarosMouZhang2022}. The second is the probabilistic approach, where the equilibrium is characterized through McKean--Vlasov forward-backward stochastic differential equations (FBSDEs) via stochastic maximum principle; see, for example, Bensoussan, Yam, and Zhang \cite{Ben-Yam-Zhang15}, Carmona and Delarue \cite{carmona-delarue13}, Buckdahn et al. \cite{BuckdahnDjehicheLiPeng09}, and Lauri\`ere and Tangpi \cite{LauriereTangpi22}. Comprehensive description of both
viewpoints and further applications of MFGs can be refered to monographs 
\cite{CarmonaDelarue18I,CarmonaDelarue18II,CardaliaguetPorretta20}.


The theory is subsequently extended from interactions through the state variables to interactions through the players' controls, or, more generally, through the joint distribution of states and controls.
Gomes, Patrizi, and Voskanyan  \cite{GomesPatriziVoskanyan14} establishes classical solvability for stationary extended mean-field game systems, while Gomes and Voskanyan \cite{GomesVoskanyan16} proves existence and give uniqueness conditions for deterministic extended mean-field games and study the associated master equation. Cardaliaguet and Lehalle \cite{CardaliaguetLehalle18} develops an optimal liquidation model with trade crowding. On the probabilistic side, Lauri\`ere and Tangpi \cite{LauriereTangpi22} obtains moment bounds for the convergence of finite-player Nash equilibria via forward-backward  propagation of chaos; Djete \cite{Djete23} characterizes measure-valued solutions as limits of approximate Nash equilibria; Possama\"i and Tangpi \cite{PossamaiTangpi25} derive non-asymptotic convergence rates using a McKean--Vlasov BSDE characterization in a weak formulation.

Our first model concerns an exponential utility portfolio game with price
impact. Price impact is central to optimal execution, liquidity crises,
portfolio liquidation, and trading under factor uncertainty; see
\cite{AlmgrenChriss01,CarlinLoboViswanathan07,HorstXiaZhou22}. Particularly, in the $N$-player formulation we consider, investor $i$ chooses its position
$\alpha_t^i$ in an risky asset with  idiosyncratic noise. The price impact is characterized by the empirical mean
\[
             m_t^N=\frac1N\sum_{j=1}^N\alpha_t^j
\]
which then influences the instantaneous return for every investor (see Section \ref{sec:price-impact-model} for details). The above model reveals the fact that the influence from each player is minor but cumulative. In the mean-field limit model, a
representative investor therefore treats a deterministic influence flow $m$
as given first, and then maximizes exponential utility of terminal wealth with a random
liability. Equilibrium requires that the prescribed flow coincides with
the mean of the optimal control $\widehat\alpha$, i.e. $m_t=\E[\widehat\alpha_t]$. 



Our model differs from earlier works in the following senses.
Carmona and Lacker \cite{CarmonaLacker15} develops a general MFGs of controls in the weak formulation, and moreover, controls are compact and the equilibrium can be characterized by
BSDEs with Lipschitz generators, in contrast to our strong formulation and mean-field qBSDEs here. Cardaliaguet and
Lehalle \cite{CardaliaguetLehalle18} study the strong formulation but with
deterministic coefficients, where the equilibrium is described through systems of PDEs. Djete \cite{Djete23} allows general  random
state-and-control interactions under additional boundedness and tightness
conditions, which we do not assume here.  Another related topici is the portfolio games with relative
performance, which are introduced by Espinosa and Touzi \cite{EspinosaTouzi15}. The mean-field model with relative performance is studied in Fu and Zhou \cite{FZ23}; while the finite-player
price impact games with relative investors are studied in B\"auerle and G\"oll \cite{BaeuerleGoll24}.

Our second model is an intermediated market for $d$ types of client-specific
risk-transfer contracts. A default-free intermediary quotes a common premium
within each type and takes the opposite side of every contract, whose payments
depend on the client's observable idiosyncratic risks. Clients maximize
exponential utility over signed positions of all contracts, and the quote is chosen so their
aggregate demand meets the intermediary's prescribed supply. 
The model shares
the endogenous-pricing mechanism of \cite{Radner68,Radner72}. For the multiple-player model, the common random quote couples clients, leading to a system of quadratic
BSDEs. Related equilibrium formulations using quadratic BSDEs include
Kardaras, Xing, and {\v Z}itkovi{\'c} \cite{KardarasXingZitkovic22}; Weston and {\v Z}itkovi{\'c} \cite{WestonZitkovic20}; Escauriaza, Schwarz, and Xing \cite{EscauriazaSchwarzXing22}. Fujii and Sekine \cite{FujiiSekine25,FujiiSekineHabit24} studies a
mean-field equilibrium pricing model with risky stocks driven by a common
Brownian noise and with agents' liabilities additionally exposed to
idiosyncratic Brownian noise. They formulate an exact market clearing condition
for their multiple-player model, but show that the risk premium process obtained from the
mean-field BSDE clears the market only in the asymptotic sense.
In our model, the representative client is driven by idiosyncratic
noise, and the equilibrium quote satisfies the mean-field clearing condition. For the multiple-player equilibria clear exactly when they exist; their convergence requires additional assumptions (see Remark~\ref{rem:pricing-finite-convergence} for details).


\subsection{qBSDEs--old, new and mean-field}

The qBSDE studied in this paper emerges when we solve the
mean-field games of controls from the above two models. For any fixed law of actions $m$, by the martingale optimality
argument, in the spirit of \cite{HuImkellerMuller05}, the feedback
control has the representation 
\[
 \widehat\alpha_t
 =\frac{Z_t}{\sigma_t}+\frac{b_t+m_t}{\eta\sigma_t^2},
\]
where $(b, \sigma)$ are stochastic coefficients from the risky asset, $\eta$ is the risk aversion parameter from the exponential utility function, and $Z$ is the solution to the BSDE when $m$ fixed (see e.g. \cite{Pham09} or Section \ref{sec:mf-martingale-rep} for the martingale optimization argument). Then by the consistency condition  
$m_t=\E[\widehat\alpha_t]$ and taking the expectation on both sides of the above identity, we have
$$
m_t= \left(1-\E[\frac{1}{\eta \sigma^2_t} ] \right)^{-1} \left(\E[\frac{Z_t}{\sigma_t}] + \E[\frac{b_t}{\eta \sigma_t^2}] \right).
$$
It follows that $m_t$ is linear in $\E[Z_t]$ when $\sigma$ deterministic, which later leaves quadratic terms of $\E[Z]^\top Z$ and $|\E[Z]|^2$ in the BSDE. The same structure arises in our second model. These examples lead us to study the mean-field BSDEs of the form
\be\label{intro-general-bsde}
Y_t=\xi+\int_t^T H_r\big(Y_r,Z_r,\ML(Y_r),\ML(Z_r),
G_r(\ML(Z_r))Z_r\big)dr-\int_t^T Z_r\,dB_r,
\ee
where $H$ is Lipschitz in $G_r(\ML(Z_r))Z_r$ and quadratic in $\ML(Z_r)$.

Linear BSDEs are introduced by Bismut \cite{Bismut73} in the study of stochastic control. Then the nonlinear theory is initiated by Pardoux and Peng \cite{PP90} and has
since become a basic tool in stochastic control, mathematical finance, and
continuous-time contract theory, as well as nonlinear PDE theory; see e.g. \cite{KPQ97,Pham09,CvitanicZhang13}. Quadratic BSDEs are
especially important for mathematical finance and physics. A major breakthrough in the one-dimensional case is achieved by Kobylanski \cite{Kobylanski00}, where a well-posedness result is obtained under the bounded terminal condition. Subsequent results further consider other cases such as unbounded terminal
values, convex generators, delayed equations, bounded Malliavin derivatives etc.; see e.g. Briand and Hu \cite{BriandHu2006PTRF,BriandHu2008PTRF}, Briand and Elie \cite{BriandElie2013SPA}, Delbaen et al. \cite{DelbaenHuRichou2011AIHPPS,DelbaenHuRichou2015DCD}, Barrieu and El Karoui \cite{BarrieuElKaroui2013AoP}, Tevzadze \cite{Tev08}, and Harter and Richou \cite{HarterRichou19}.

The multidimensional problem is substantially more delicate and the general case remains open  (see Peng \cite{Peng1999}, Espinosa and Touzi \cite{EspinosaTouzi15}, and Jackson \cite{Jack23}). The counterexample of Frei and dos Reis
\cite{FreiDosReis11} shows that a general multidimensional quadratic system
need not possess a solution, implying the necessity of some structure or
smallness assumptions. Solvability is
known for special structures, including the classes of Cheridito and Nam
\cite{CheriditoNam15}, diagonally quadratic systems
\cite{HuTang16,FanHuTang23}, Markovian cases \cite{XingZitkovic18}, and triangular systems
\cite{Luo20,JacksonZitkovic22,FanHuTang25}. In the recent work \cite{HorstSchmidekZhang26}, authors study a  weakly interactive case, where they show the well-posedness under a smallness assumption but independent of dimension, for the purpose of passing to the mean-field limit. More precisely, they prove multiple-player convergence under bounded data and a smallness condition. However, in this paper we study the limit for unbounded data, using polynomial Malliavin assumptions locally and entropy estimates globally, without a smallness condition on the interaction.

Mean-field BSDEs, in which the generator depends on the law of the solution,
arise from limits of weakly interacting backward systems in Buckdahn et al.
\cite{BuckdahnDjehicheLiPeng09}.  Their quadratic variants are subsequently
studied under several forms: Hao, Wen and Xiong 
\cite{HaoWenXiong22} considers a general quadratic system under smallness condition on the terminal condition; Hibon Hu and Tang \cite{HibonHuTang23} studies the local solution when the driver is quadratic in $Z$ but subquadratic in $\E[Z]$; the recent work Hao et al. \cite{HaoHuTangWen25} largely extends earlier local results to global in several cases. Via a $Y$ independent and separable form, they allow a quadratic term like $|\E[Z]|^2$, while the separable structure excludes the mixed $Z\E[Z]$ terms. More recently, Ding, Nam and Wen \cite{DingNamWen25} studies the case when generators have general quadratic interaction between $Z$ and an independent copy of $Z$. In their formulation, they allow a term of $Z \E[Z]$, but use bounded terminal data, with structural regularity in the main theorems and a smallness condition on the centered terminal value in the extension without this regularity. Moreover, the other quadratic term of kind $|\E[Z]|^2$ seems out of reach.

This paper tackles the challenging case where the quadratic part comes from $\E[Z]$. Unlike earlier works on mean-field qBSDEs, which are motivated by classical qBSDEs and focus on the quadratic growth in $Z$, the quadratic terms in the mean-field limiting system comes from a weak interaction of controls from all players, and thus are represented by a generator quadratic in $\E[Z]$. More precisely, under suitable propagation of chaos and stability conditions, the empirical interactions in the multiple-player equilibrium are expected to converge to deterministic first moment terms involving $\E[Z_t]$, while the off-diagonal components associated with other players' idiosyncratic noises disappear. After imposing the equilibrium consistency condition, this produces the quadratic terms of types $\E[Z_t]^\top Z_t$ and $|\E[Z_t]|^2$. In particular, in the limiting mean-field model, since there is no common noise, the representative agent's idiosyncratic risk is hedgeable, but with a consistency condition price to pay.

To solve \eqref{intro-general-bsde} with driver including $\E[Z]^\top Z$ term, one may expect to freeze the measure variables of \eqref{intro-general-bsde} and downgrade the equation to a BSDE with a stochastic-Lipschitz generator. However, in view of \cite{BC08}, the genenral stochastic Lipschitz coefficients would cause a degeneracy in integrability of solutions. More precisely, for any such $L^p$-coefficients, one should expect the solution to be $L^{q}$ with some $q<p.$ Our first key observation is that when the stochastic coefficient is determinisitic, we do not have this degeneracy in integrability as long as coefficients has enough integrability. This special structure allows us to work with
possibly unbounded terminal values under polynomially integrable Malliavin derivative conditions,
without assuming exponential integrability of $\xi$ or an a priori BMO
bound for the local solution. To handle the term $|\E[Z_t]|^2$, we exploits precisely the addibility structure of generators from the mean-field models. Indeed, we decompose
$Y=Y^{(1)}+Y^{(2)}$: the pair $(Y^{(1)},Z)$ solves the kind of BSDEs with stochastic-Lipschitz driver after freezing the measure, while the deterministic
finite-variation process $Y^{(2)}$ carries the law-dependent quadratic
term. The resulting BSDE--ODE system can be shown to be equivalent to the original equation under mild assumptions. Finally, to extend the local solution to a global one, we add an extra exponential integrability on its Malliavin derivatives, and exploits the entropy inequality to obtain a global apriori bound on $Y$ and its Malliavin derivative $DY$.

The main contributions of the paper are three folds. Firstly we introduce a general mean-field portfolio game with price impact, and a finite-contract pricing model where the risks depend on idyosyncratic noise, which we connect to a new kind of mean-field qBSDEs. Secondly, we develop general methods to solve this kind of qBSDEs locally under polynomial integrability conditions on the terminal data and their Malliavin derivatives, and globally under a exponential integrability condition on the Malliavin derivatives, using a BSDE--ODE decomposition and entropy estimates. Thirdly, we show the existence and uniqueness of mean-field equilibria for the two mean-field models by applying the qBSDE theory.

The remainder of the paper is organized as follows.  After the notation and
preliminaries below, Section~\ref{sec:utility-model} presents the mean-field
price-impact and finite-contract pricing models, and moreover, derives the characterization of their
equilibria via mean-field qBSDEs.  Section~\ref{sec:local-solvability} establishes local
well-posedness, first for the $\E[Z]Z$ interaction and then for the additional
$|\E[Z]|^2$ term.  Section~\ref{sec:global} globalizes the local solution
under the exponential integrability condition on the Malliavin derivatives.
Section~\ref{sec:utility-application} applies these results to
establish existence and uniqueness of equilibria in the two mean-field models.
Section~\ref{sec:conclusion} concludes.

\noindent\textbf{Acknowledgement.}
The author gratefully acknowledges financial support from the Fundamental Research Funds for the Central
Universities and the NSF of Shandong (ZR2023MA026); DFG CRC/TRR 388 ``Rough Analysis, Stochastic
Dynamics and Related Fields''. This work is started when HZ works in Humboldt University Berlin. Special thank to Ulrich Horst and Emil Schmidek for helpful discussion on related topics.

\subsection{Setups and notations}

Suppose that $(\Omega, \MF, P)$ is a probability space with $d$-dimensional Brownian motion $B$ and $\{\MF_t\}_{t\ge 0}$ is the augmented Brownian filtration. In this paper we fix time interval as $[0,T]$. Denote by $V$ a Banach space with norm $|\cdot|_V$, and if $V$ is Euclidean, we write $|\cdot|_V$ by $|\cdot|.$ Throughout, vectors in $\R^m$ are regarded as column vectors.  For a vector
or matrix $A$, the notation $A^\top$ denotes its transpose For any $p \ge 1,$ denote by
\begin{itemize}[leftmargin=1.2em, labelsep=0.4em]
\item For an $N$-dimensional Euclidean vector $\boldsymbol\alpha=(\alpha^1,..., \alpha^N)$, we write $\boldsymbol\alpha^{-i}:=(\alpha^1,...\alpha^{i-1}, \alpha^{i+1},..., \alpha^N).$ We also write $(\beta^i; \boldsymbol\alpha^{-i})=(\alpha^1,...\alpha^{i-1}, \beta^i, \alpha^{i+1},..., \alpha^N)$ for any $\beta^i \in \R;$

\item $\MC ( [0,T], V)$ the space of continuous $V$-valued paths $y$ on $[0,T]$ with norm
$\| y \|_{[0,T] }:= \sup_{t\in[0,T]} |y_t|_V,$

\item $L^p([0,T],V)$ the space of Lebesgue measurable function $a$ on $[0,T]$ with bounded 
$\| a \|_{L^p}^p:=\int_0^T |a_r|^p_V dr,$

\item $\MC^k(\R^n,\R^m)$ the space of functions from $\R^n$ to $\R^m$ with continuous derivatives up to $k$-th order, and $\MC^k_b(\R^n,\R^m)$ the subspace of $\MC^k(\R^n,\R^m)$ such that the function $f$ and its all derivatives are bounded, with norm 
$\|f\|_{\MC_b^k}:= \sum_{0\le i \le k} \sup_{x\in \R^n}|\partial_x^i f(x)|$.
For $k=0$, we write $\MC_b:=\MC_b^0$ and
$\|f\|_{\MC_b}:=\sup_{x\in\R^n}|f(x)|$;

\item $\ML^p( \R^m)$ the space of $\mathcal{F}$-measurable $\R^m$-valued random variables $\xi$ with
$\|\xi \|_{\ML^p}^p:= \E|\xi |^p< \infty,$


\item $\MM^p(\R^m) $ the space of $\R^m$-valued predictable processes $Z$ with $ \|Z \|_{\MM^p}^p:=\E [ \int_0^T |Z_r|^p dr] < \infty,$


\item $\MH^p( \R^m)$ the space of $\R^m$-valued predictable processes $Z$ with $ \|Z \|_{\MH^p}^p:=\E [ \int_0^T |Z_r|^2 dr]^{\frac{p}{2}} < \infty.$ Note that, on the finite time interval $[0,T]$, we have $\MM^p \subseteq \MH^p$ for any $p \ge 2,$ and $\MH^2=\MM^2.$

\item $\M^p(\R^m) $ the space of $\R^m$-valued progressive measurable processes $Z$ with\\
 $ \|Z \|_{\M^p}^p:=\E [ \int_0^T |Z_r| dr]^p  < \infty,$

\item $\MS^p(\R^m)$ the space of continuous adapted processes $Y$ with $\| Y \|_{\MS^p}^p:=\E[\sup_{t\in[0,T]}| Y|^p]< \infty.$
\item $\op_1(\R^m)$ the space of probability measures on $\R^m$ with finite first order moment. We equip this space with W-$1$ Wasserstein distance
$$
W_1(\mu,\nu):= \inf_{\Pi(\mu,\nu)} \int_{\R^m \times \R^m} |x-y|\ \Pi(dx,dy),
$$where the infimum is taken over all Borel probability measures on $\R^m \times \R^m$ with marginal measures given by $\mu$ and $\nu.$ For any $p \ge 1,$ we also write the $p$-th moment of $\mu \in \op_1$ as the following whenever it is finite,
$$
\| \mu\|_p:= \int_{\R^m} |x|^p \mu(dx).
$$

%

\end{itemize}

In the following we usually omit the obvious dimension index in the above notations for simplicity when no confusion is raised, i.e. $\MH^p$ short for $\MH^p(\R^d),$ and always omit the notation $\R^m$ when $m=1.$  For any local martingale $M,$ we denote by $\mathcal{E}(M).:= \exp(M_.- \frac12\langle M \rangle_.)$ the stochastic exponential of $M$, and write $\E_t^{Q}[\cdot]$ short for the conditional expectation $\E^Q[\cdot|\MF_t]$ under the probability measure $Q.$ The upper index $Q$ will be omitted if $Q=P.$ We denote by $\MH^p([t,s],\R^m)$ the space of processes $X$ restricted on $[t,s]$ with $\| X\|_{\MH^p,{[t,s]}}^p:=\E [ \int_t^s |X_r|^2 dr]^{\frac{p}{2}}< \infty,$ and notations like $\MS^p{([t,s])},$ $\MH^2_{\text{BMO}}([t,s])$ are defined in the similar way. We write $\lesssim_C$ to imply that an ineuqality holds up to the right hand side multiplied by a constant depending on $C.$

Let us briefly review some basic notations and results about Malliavin calculus on the Wiener space. We refer to \cite{N06} for details on this theory. Denote by $\MS$ the space of random variables in the form
\begin{equation*}
\xi =F(\int_0^T h^1_r   dW_r, \cdots,  \int_0^T h^k_r   dW_r)
\end{equation*}
where $h^1,\cdots, h^k \in L^2([0,T], \R^d),$ $F\in \MC^{\infty}_{b}(\R^k, \R)$. For any $\xi \in \MS,$ its Malliavin derivative is defined as the $d$-dimensional process
\begin{equation*}
D_u \xi = \sum_{j=1}^k \partial_j F(\int_0^T h^1_r  dW_r, \cdots, \int_0^T h^k_r   dW_r) h_u^j, \ \ \ u\in [0,T].
\end{equation*}
For any $\xi \in \MS,$ let 
\begin{equation*}
\| \xi \|_{\D^{1,2}}^2:= \E\left[ |\xi|^2 + \int_0^T |D_r \xi|^2 dr \right].
\end{equation*}
Then the operator $D$ is closable and we denote by $\D$ the completion of $\MS$ under norm $\|\cdot \|_{\D^{1,2}}.$
Note that if $\xi \in \MF_t,$ then $D_u \xi =0 $ for any $u \in (t,T].$ Let $\bL^{1,2}([0,T],\R^m)$ the space of progressively measurable process $A$ such that
\begin{itemize}
	\item[(i)]  $A_t \in (\D^{1,2})^m $ for a.e. $t\in [0,T];$ 
	\item[(ii)]  the process $(t,\ome) \rightarrow D_.A(t,\ome) \in L^2([0,T],\R^{d \times m}) $ admits a progressively measurable version;
	\item[(iii)] $\| A \|_{\bL^{1,2}}^2:= \E\left[ \int_0^T |A_r|^2 dr + \int_0^T \int_0^T |D_u A_r |^2 du dr \right] < \infty.$

\end{itemize}
We shall use the following standard random-field chain rule.  Let
$F:\Omega\times\R^m\to\R^k$ be such that $F(x)\in\D^{1,2}$ for every
$x$, $F(\omega,\cdot)$ is continuously differentiable with uniformly bounded spatial
derivative, and $D_uF(\cdot)$ admits a jointly measurable version such that
$D_uF(\cdot)\in\MC_b(\R^m,\R^{k\times d})$ for $du\otimes dP$-a.e.
$(u,\omega)$ and
\[
 \E\int_0^T
 \|D_uF(\cdot)\|_{\MC_b(\R^m,\R^{k\times d})}^2du<\infty.
\]
If $X\in(\D^{1,2})^m$, then
\be\label{random-field-chain-rule}
 D_u[F(X)]=D_uF(X)+\nabla F(X)D_uX.
\ee
The process-valued version holds under the analogous assumptions on
$[0,T]\times\Omega$ and the condition
\[
 \E\int_0^T\int_0^T
 \|D_uF_r(\cdot)\|_{\MC_b(\R^m,\R^{k\times d})}^2du\,dr<\infty.
\]

\section{Mean-field exponential utility equilibrium models}\label{sec:utility-model}

In this section, we present the equilibrium models that motivate the class of
quadratic BSDEs studied below.  We begin with a benchmark price impact
portfolio game.  

%

\subsection{A benchmark price impact portfolio game}
\label{sec:price-impact-model}

\subsubsection{Multiple-player motivation and mean-field limit}

Consider $N$ agents with common preferences on a probability space supporting $N$ 
independent $1$-dimensional Brownian motions $B^1,\ldots,B^N$ with the augmented filtration generated by $B^i$ denoted by $( \MF^{B^i}_t)_t$.  Agent $i$ has initial wealth
$x^i$, a terminal liability $\xi^i$, and trades an agent-specific risky
asset with return
\begin{equation*}
 \frac{dS_t^{i,N}}{S_t^{i,N}}
 =(b_t^i+ \lambda^i m_t^N)dt+\sigma_t^i dB_t^i, \qquad i=1,...,N,
\end{equation*}
where $(b^i,\sigma^i)$ are idiosyncratic market coefficients, $\lambda^i$ is a deterministic parameter describing price impact sensitivity, and $m^N$ is the empirical price impact. If $\alpha_t^i$ is the dollar of agent $i$ invested in the $i$'s risky
asset, and takes values in some admissible control set $\MA$. Then the corresponding wealth of agent $i$ is
\begin{equation*}
 X_t^{i,N;\boldsymbol\alpha}
 =x^i+\int_0^t\alpha_r^i(b_r^i+ \lambda^i m_r^N)dr
      +\int_0^t\alpha_r^i\sigma_r^i dB_r^i.
\end{equation*}
With a risk aversion $\eta^i >0$, agent $i$ evaluates an investment
$\alpha^i$, given others' investment $\boldsymbol\alpha^{-i}$, through
\begin{equation*}
 J_i^N(\alpha^i;\boldsymbol\alpha^{-i})
 :=\E\left[-\exp\left(-\eta^i
   (X_T^{i,N;(\alpha^i,\boldsymbol\alpha^{-i})}-\xi^i)\right)\right].
\end{equation*}
A Nash equilibrium is a profile
$\widehat{\boldsymbol\alpha}^{N}$ satisfying
\begin{equation*}
 J_i^N(\widehat\alpha^{i,N};
       \widehat{\boldsymbol\alpha}^{-i,N})
 =\sup_{\alpha^i \in \MA}
   J_i^N(\alpha^i;\widehat{\boldsymbol\alpha}^{-i,N}),
 \qquad i=1,\ldots,N.
\end{equation*}

Let $(B^i,\xi^i,b^i,\sigma^i)_{1\le i\le N}$ are independent copies: more
precisely, $\xi^i \in \MF^{B^i}_T$, and $b^i$ and $\sigma^i$ are progressively measurable with respect to $(\MF^{B^i}_t)_t$. Suppose all agents are all similar and share the same price impact sensitivity $\lambda^i  \equiv 1$ and risk aversion $\eta>0$. For a strategy 
$\boldsymbol\alpha=(\alpha^1,\ldots,\alpha^N)$, let the empirical price
impact given by the averaged strategies of all agents
\be\label{empirical-impact}
m_t^N:=\frac{1}{N}\sum_{j=1}^N\alpha_t^j.
\ee

In this case, the propagation of chaos implies that, in an appropriate topology and under suitable condittions,
\[
 m_t^N=\frac1N\sum_{j=1}^N\widehat\alpha_t^{j,N}
 \longrightarrow m_t:=\E[\widehat\alpha_t].
\]
For more details of the multiple-player model (even with a common noise) and its mean-field limit, see \cite{HorstSchmidekZhang26}.

We now formulate the corresponding mean-field problem.  Let $(B,\xi,b,\sigma)$ be a generic
copy of the above idiosyncratic data, and $b,\sigma$ are predictable in the
Brownian filtration.  We assume that $b$, $\sigma$, and $\sigma^{-1}$ are uniformly 
bounded.  Given a deterministic influence flow
$m\in L^2([0,T])$, the representative risky
asset has dynamics
\begin{equation*}
\frac{dS_t^m}{S_t^m}=(b_t+m_t)dt+\sigma_t dB_t.
\end{equation*}
Let $\alpha_t$ denotes the dollar amount invested in the representative risky
asset, the wealth process starting from $x\in\R$ is
\begin{equation*}
X_t^{\alpha,m}
=x+\int_0^t\alpha_r(b_r+m_r)dr
  +\int_0^t\alpha_r\sigma_r dB_r.
\end{equation*}
The representative agent has constant risk aversion parameter $\eta>0$ and
a terminal liability $\xi\in\MF_T$.  The agent tries to maximise its utility 
\begin{equation*}
\sup_{\alpha\in\MA }J(\alpha;m),\qquad
J(\alpha;m):=\E\left[-\exp\left(-\eta
  (X_T^{\alpha,m}-\xi)\right)\right].
\end{equation*}
A mean-field equilibrium is a pair $(\widehat\alpha,m)$ such that
\be\label{mf-equilibrium-condition}
J(\widehat\alpha;m)=\sup_{\alpha\in\MA }J(\alpha;m),
\qquad m_t=\E[\widehat\alpha_t]\quad\text{for a.e. }t\in[0,T].
\ee

\subsubsection{The mean-field equilibrium via the martingale optimality principle}\label{sec:mf-martingale-rep}

Fix any fixed deterministic impact flow $m$, let $(Y,Z)$ solve
the BSDE
\begin{equation*}
Y_t=\xi+\int_t^T f_r^m(Z_r)dr-\int_t^T Z_r dB_r,
\end{equation*}
where the driver is to be determined.  For every admissible control $\alpha$,
consider the candidate value process
\begin{equation*}
R_t^{\alpha,m}:=-\exp\left(-\eta
  (X_t^{\alpha,m}-Y_t)\right).
\end{equation*}
An application of It\^o's formula shows that the drift part of
$dR^{\alpha,m}$ is
\be\label{candidate-value-drift}
R_t^{\alpha,m}\left\{-\eta\big(\alpha_t(b_t+m_t)
  +f_t^m(Z_t)\big)
  +\frac{\eta^2}{2}|\alpha_t\sigma_t-Z_t|^2\right\}dt.
\ee
Since $R^{\alpha,m}<0$, the martingale optimality principle is satisfied
if the expression in the above bracket is nonnegative for every $\alpha$ and vanishes
at the optimizer.  Accordingly, we set
\begin{align}
f_t^m(z)
:=\inf_{a\in\R}\left\{\frac{\eta}{2}|a\sigma_t-z|^2
  -a(b_t+m_t)\right\} =-\frac{b_t+m_t}{\sigma_t}z
  -\frac{|b_t+m_t|^2}{2\eta\sigma_t^2},
\label{best-response-driver}
\end{align}
and the candidate optimal control is
\be\label{best-response-control}
\widehat\alpha_t
=\frac{Z_t}{\sigma_t}
 +\frac{b_t+m_t}{\eta\sigma_t^2}.
\ee
In this case, the standard martingale optimality argument (see e.g. \cite{Pham09}) verifies that it is optimal whenever the candidate control is admissible, and in particular,
\be\label{utility-value}
\sup_{\alpha\in\MA }J(\alpha;m)
=-\exp\big(-\eta(x-Y_0)\big).
\ee

It remains to impose the consistency condition in
\eqref{mf-equilibrium-condition}.  Taking expectations in
\eqref{best-response-control} we have
\be\label{impact-fixed-point}
m_t=\E\left[\frac{Z_t}{\sigma_t}\right]
  +\frac{1}{\eta}\E\left[\frac{b_t}{\sigma_t^2}\right]
  +\frac{m_t}{\eta}\E\left[\frac{1}{\sigma_t^2}\right].
\ee
Let
\be\label{no:beta-q}
 a_t:=\E[\sigma_t^{-2}],\qquad
 \beta_t:=\E[b_t\sigma_t^{-2}],\qquad
 q_t(Z):=\E[Z_t/\sigma_t].
\ee
Assume that there is a constant $\delta>0$ such that
\be\label{impact-nondegeneracy}
1-\frac{a_t}{\eta}\ge\delta,
\qquad\text{for a.e. }t\in[0,T],
\ee
and define
\begin{equation}\label{no:c}
c_t:=\left(1-\frac{a_t}{\eta}\right)^{-1}.
\end{equation}
Then by \eqref{impact-fixed-point}, we obtain
\be\label{equilibrium-impact}
m_t=c_t\left(q_t(Z)+\frac{\beta_t}{\eta}\right).
\ee
Substituting this identity into \eqref{best-response-driver} we obtain the
following mean-field quadratic BSDE
\begin{align}
Y_t={}&\xi+\int_t^T\left[
 -\frac{b_r+c_r(q_r(Z)+\beta_r/\eta)}{\sigma_r}Z_r
 -\frac{\bigl[b_r+c_r(q_r(Z)+\beta_r/\eta)\bigr]^2}
        {2\eta\sigma_r^2}
 \right]dr
 -\int_t^T Z_r dB_r.
\label{stochastic-utility-mf-qbsde}
\end{align}
%

An important case is that  $\sigma$ is deterministic. Note that in this case
$
 m_t=c_t\left(\frac{\E[Z_t]}{\sigma_t}
  +\frac{\E[b_t]}{\eta\sigma_t^2}\right).
$
Consequently, \eqref{stochastic-utility-mf-qbsde} becomes
\be
\begin{split}
Y_t=&\ \xi+\int_t^T \Big[ -\frac{c_r}{\sigma_r^2}\E[Z_r]Z_r - (\frac{b_r}{\sigma_r}
  +\frac{c_r\E[b_r]}{\eta\sigma_r^3} )Z_r\\
 &-\frac{1}{2\eta\sigma_r^2}
   \Big| b_r+\frac{c_r\E[b_r]}{\eta\sigma_r^2}
  +\frac{c_r\E[Z_r]}{\sigma_r} \Big|^2
  \Big]dr-\int_t^T Z_r dB_r.
\label{utility-mf-qbsde}
\end{split}
\ee

\begin{rem}
  Note that in the above model, we do not assume a common noise, and thus the corresponding multiple-player model is driven by indiosyncratic noises. With a common noise, both the multiple-player model and the mean-field limit can be formulated in a similar way, but a pathwise quadratic term $Z^2$ will appear. 
  The global result of Section~\ref{sec:global} does not cover this additional pathwise quadratic term  (see \cite{FreiDosReis11} for a counterexample of global well-posedness). We refer to \cite{HorstSchmidekZhang26} for a detailed discussion on this topic.

\end{rem}

We finally make explicit the connection between
\eqref{utility-mf-qbsde} and the equations analyzed in the remainder of the
paper.  For $\rho\in\op_1(\R)$, write
$G_t(\rho):=  \bar z(\rho):=\int_{\R}u\rho(du)$ and set
\begin{align*}
&F_t(y,z,\nu,\rho,\MZ):=-\frac{c_t}{\sigma_t^2}\MZ
  -\left(\frac{b_t}{\sigma_t}
    +\frac{c_t\E[b_t]}{\eta\sigma_t^3}\right)z
  -\frac{c_t}{\eta\sigma_t^3}
    \left(b_t+\frac{c_t\E[b_t]}{\eta\sigma_t^2}\right)\bar z(\rho)
  -\frac{1}{2\eta\sigma_t^2}
    \left|b_t+\frac{c_t\E[b_t]}{\eta\sigma_t^2}\right|^2\\
    &  h_t(\nu,\rho):=-\frac{c_t^2 \bar{z}(\rho)^2 }{2\eta\sigma_t^4}, \ \ \ H_t:= F_t+h_t.
\end{align*}
Then \eqref{utility-mf-qbsde} is precisely
\be\label{utility-mf-law-form}
Y_t=\xi+\int_t^T H_r\big(Y_r,Z_r,\ML(Y_r),\ML(Z_r),
  G_r(\ML(Z_r))Z_r\big)dr-\int_t^T Z_r dB_r.
\ee

\subsection{A mean-field finite-contract pricing model with idiosyncratic risks}
\label{sec:pricing-model}

A default-free intermediary offers $d$ types of client-specific risk-transfer
contracts. Within each type, contracts have a common notional unit and premium
rate, while payments depend on each client's observable idiosyncratic risk
factors. The intermediary takes the opposite side of every contract.

Let $(B^i,\xi^i)_{i=1}^N$ be independent copies of a $d$-dimensional Brownian
motion and a terminal liability measurable with respect to its filtration,
and let $\mathbb F^N$ be their augmented joint Brownian filtration.
In discounted units, the cumulative gains per unit notional satisfy
\be\label{finite-contract-prices}
 dG_t^{i,\nu^N}=\nu_t^Ndt+\Sigma_t dB_t^i,
 \qquad G_0^{i,\nu^N}=0,\qquad i=1,\ldots,N,
\ee
where $\nu^N$ is $\mathbb F^N$-predictable with values in $\R^d$ and $\Sigma$ is deterministic and
invertible, with $\Sigma$ and $\Sigma^{-1}$ bounded. The quoted premium rate is
$p^N:=-\nu^N$: a positive position pays the premium and receives the indexed
payment $\Sigma_t dB_t^i$.

For an $\mathbb F^N$-predictable vector of signed notionals $\alpha^i$, client
$i$'s wealth and exponential utility, with risk aversion $\eta>0$, are
\begin{equation*}
 X_t^{i,\alpha^i,\nu^N}
 =x^i+\int_0^t(\alpha_r^i)^\top\nu_r^Ndr
      +\int_0^t(\alpha_r^i)^\top\Sigma_r dB_r^i,
 \qquad
 J_i^N(\alpha^i;\nu^N)
 :=\E\left[-e^{-\eta(X_T^{i,\alpha^i,\nu^N}-\xi^i)}\right].
\end{equation*}
The intermediary's supply fixes aggregate signed positions $Ns_t$, for
a given bounded measurable $s:[0,T]\to\R^d$. Define per-client demand by
\be\label{finite-contract-demand}
 A_t^N:=\frac1N\sum_{i=1}^N\alpha_t^i,
\ee
so the supply condition is $A^N=s$. The intermediary's wealth satisfies
\[
 dW_t^{D,N}=-\sum_{i=1}^N(\alpha_t^i)^\top dG_t^{i,\nu^N},
 \qquad dW_t^{D,N}+\sum_{i=1}^N dX_t^{i,\alpha^i,\nu^N}=0.
\]
Then clients take the common quote as given, and the quote
adjusts to meet the supply mandate. Write $\MA_i^N(\nu^N)$ for client $i$'s
admissible portfolios.

\begin{defi}[Multiple-player finite-contract pricing equilibrium]
\label{def:finite-client-pricing-equilibrium}
For fixed $N$, a pair $(\nu^N,\widehat{\boldsymbol\alpha}^{N})$ is called an
$N$-player pricing equilibrium if $\nu^N$ is an admissible
$\mathbb F^N$-predictable process and
$\widehat{\boldsymbol\alpha}^{N}
=(\widehat\alpha^{1,N},\ldots,\widehat\alpha^{N,N})$ with
$\widehat\alpha^{i,N}\in\MA_i^N(\nu^N)$. Moreover,
\[
 J_i^N(\widehat\alpha^{i,N};\nu^N)
 =\sup_{\alpha^i\in\MA_i^N(\nu^N)}J_i^N(\alpha^i;\nu^N),
 \qquad i=1,\ldots,N,
\]
and aggregate demand meets the prescribed supply:
\be\label{eq:exact-clearing}
 \frac1N\sum_{i=1}^N\widehat\alpha_t^{i,N}=s_t
 \quad dt\otimes d\PP \text{-a.e.}
\ee
\end{defi}

The common quote couples clients through \eqref{eq:exact-clearing}.
For fixed $\nu^N$, set $\theta^N:=\Sigma^{-1}\nu^N$.
Applying It\^o's formula to
$-\exp[-\eta(X_t^{i,\alpha^i,\nu^N}-Y_t^{i,N})]$ gives the candidate BSDE driver,
for $\boldsymbol z=(z^1,\ldots,z^N)$,
\begin{align}
 f_t^{i,N,\nu^N}(\boldsymbol z)
 &:=\inf_{a\in\R^d}\left\{
   \frac{\eta}{2}\left(
      |\Sigma_t^\top a-z^i|^2+\sum_{j\ne i}|z^j|^2
   \right)-a^\top\nu_t^N\right\} \notag\\
 &=-\theta_t^N\cdot z^i-
   \frac{|\theta_t^N|^2}{2\eta}
   +\frac{\eta}{2}\sum_{j\ne i}|z^j|^2,
 \label{finite-client-fixed-driver}
\end{align}
and the minimizing portfolio is
\be\label{finite-client-fixed-control}
 \widehat\alpha_t^{i,N}
 =(\Sigma_t^\top)^{-1}\left(
 Z_t^{i,i,N}+\frac{\theta_t^N}{\eta}\right).
\ee
This identifies an optimal portfolio provided that the BSDE
\be\label{finite-client-fixed-bsde}
 Y_t^{i,N}=\xi^i+\int_t^T
 f_r^{i,N,\nu^N}\bigl((Z_r^{i,j,N})_{j=1}^N\bigr)dr
 -\sum_{j=1}^N\int_t^T Z_r^{i,j,N}dB_r^j,
\ee
admits a solution satisfying the martingale verification conditions.
Denote by $\overline Z_t^N:=N^{-1}\sum_{i=1}^NZ_t^{i,i,N}$, and we have by the supply condition
and \eqref{finite-client-fixed-control} that
\[
 \theta_t^N=\eta(\Sigma_t^\top s_t-\overline Z_t^N),
 \qquad
 \nu_t^N=\eta\Sigma_t(\Sigma_t^\top s_t-\overline Z_t^N).
\]
Substituting the above into \eqref{finite-client-fixed-bsde}, we obtain the candidate equilibrium system
\be \label{finite-client-equilibrium-bsde}
\begin{split}
 Y_t^{i,N}={}&\xi^i+\int_t^T\left[
   \eta(\overline Z_r^N-\Sigma_r^\top s_r)\cdot Z_r^{i,i,N}
   -\frac{\eta}{2}|\overline Z_r^N-\Sigma_r^\top s_r|^2
   +\frac{\eta}{2}\sum_{j\ne i}|Z_r^{i,j,N}|^2
 \right]dr \\
 &-\sum_{j=1}^N\int_t^T Z_r^{i,j,N}dB_r^j,
 \qquad i=1,\ldots,N.
\end{split}
\ee
The off-diagonal term represents other clients' risks transmitted through the
random quote, which client $i$ cannot hedge using contracts driven by $B^i$.
It disappears in the representative-client problem with a deterministic quote.
Formally, the mean-field limit is
\be\label{finite-contract-lln}
 \nu^N\longrightarrow\nu^*,
 \qquad
 \frac1N\sum_{i=1}^N\widehat\alpha^{i,N}
 =s=\E[\widehat\alpha].
\ee
Here average demand already equals $s$ for every $N$; convergence of the quote
requires additional estimates, discussed in
Remark~\ref{rem:pricing-finite-convergence}.

Now we formulate the corresponding mean-field problem.  Let $(B,\xi)$ be a generic copy of the above idiosyncratic data, and $B$ is a $d$-dimensional Brownian motion.  We assume that $\Sigma$ and $\Sigma^{-1}$ are bounded. For a representative client with data $(B,\xi)$ and a deterministic quote
$\nu\in L^q([0,T],\R^d)$, the wealth process is
\begin{equation*}
 X_t^{\alpha,\nu}
 =x+\int_0^t\alpha_r^\top\nu_rdr
      +\int_0^t\alpha_r^\top\Sigma_r dB_r.
\end{equation*}
Given $\nu$, the representative client maximizes
\begin{equation*}
 J(\alpha;\nu):=
 \E\left[-\exp\left(-\eta(X_T^{\alpha,\nu}-\xi)\right)\right]
\end{equation*}
over the admissible portfolios $\MA(\nu)$ defined in
Section~\ref{sec:pricing-application}.

A mean-field pricing equilibrium is a pair $(\widehat\alpha,\nu^*)$ with
deterministic $\nu^* \in L^q([0,T],\R^d)$ and
$\widehat\alpha\in\MA(\nu^*)$ such that
\be\label{eq:mean-field-pricing-equilibrium}
 J(\widehat\alpha;\nu^*)
 =\sup_{\alpha\in\MA(\nu^*)}J(\alpha;\nu^*),
 \qquad
 \E[\widehat\alpha_t]=s_t
 \quad\text{for a.e. }t\in[0,T].
\ee


For fixed $\nu$, set $\theta_t^\nu:=\Sigma_t^{-1}\nu_t$.
By the martingale optimality principle, we see that
\begin{align}
 f_t^\nu(z)
 &=\inf_{a\in\R^d}\left\{
    \frac{\eta}{2}|\Sigma_t^\top a-z|^2-a^\top\nu_t\right\}
 =-\theta_t^\nu z-\frac{|\theta_t^\nu|^2}{2\eta},
 \label{pricing-fixed-driver}\\
 \widehat\alpha_t^\nu
 &=(\Sigma_t^\top)^{-1}\left(Z_t+\frac{\theta_t^\nu}{\eta}\right).
 \label{pricing-fixed-control}
\end{align}
By the supply condition and
\eqref{pricing-fixed-control} we have
\be\label{pricing-return}
 \theta_t^*=\eta(\Sigma_t^\top s_t-\E[Z_t]),
 \qquad
 \nu_t^*=\eta\Sigma_t(\Sigma_t^\top s_t-\E[Z_t]).
\ee
Substituting the above into \eqref{pricing-fixed-driver}, we obtain the mean-field qBSDE
\be\label{pricing-mf-qbsde}
 Y_t=\xi+\int_t^T\left[
   \eta(\E[Z_r]-\Sigma_r^\top s_r)Z_r
   -\frac{\eta}{2}|\E[Z_r]-\Sigma_r^\top s_r|^2
 \right]dr-\int_t^T Z_r dB_r.
\ee
Once \eqref{pricing-mf-qbsde} is solved, the equilibrium candidates are
recovered by
\begin{equation*}
 \nu_t^*=\eta\Sigma_t(\Sigma_t^\top s_t-\E[Z_t]),
 \qquad
 \widehat\alpha_t
 =s_t+(\Sigma_t^\top)^{-1}(Z_t-\E[Z_t]).
\end{equation*}
Thus the premium $p_t^*=-\nu_t^*=\eta\Sigma_t(\E[Z_t]-\Sigma_t^\top s_t)$
reflects aggregate hedging demand relative to supply, while individual positions adjust according to the deviation of $Z_t$ from its expectation.
Equation~\eqref{pricing-mf-qbsde} has precisely the law form
\eqref{utility-mf-law-form} with
\[
 G_t(\rho)=\bar z(\rho)-\Sigma_t^\top s_t,\qquad
 F_t(y,z,\nu,\rho,\MZ)=\eta \MZ,
 \qquad h_t(\nu,\rho)=-\frac{\eta}{2}|\bar z(\rho)-\Sigma_t^\top s_t|^2.
\]

For example, take $d=1$, $\Sigma_t=\sigma>0$, constant $s$, and liability $\xi=\kappa B_T+\frac{\gamma}{2}B_T^2$, with $\gamma>0$.
By the ansatz \eqref{pricing-mf-qbsde}, we have
\[
 p_t^*=\eta\sigma(\kappa-\sigma s)e^{\eta\gamma(T-t)},
 \qquad \widehat\alpha_t=s+\frac{\gamma}{\sigma}B_t.
\]
Then the premium increases with the linear liability coefficient $\kappa$ and decreases with
supply $s$; when $\kappa>\sigma s$, greater liability convexity (larger $\gamma$) increases the equilibrium premium at each $t<T$.
Later we see that this example satisfies Assumption~\ref{assu:pricing-equilibrium} for our main theorem.

\section{Local solvability of mean-field quadratic BSDEs}\label{sec:local-solvability}

Throughout Sections~\ref{sec:local-solvability} and~\ref{sec:global}, the
Brownian motion have any fixed finite dimension $d\ge1$.  Accordingly,
$Z_t$, $a_t$, and $G_t(\ML(Z_t))$ take values in $\R^d$, and products of two
such vectors, for example $a_tZ_t$ and
$G_t(\ML(Z_t))Z_t$, denote Euclidean inner products.  Malliavin derivatives
are understood componentwise in the Brownian direction.

Motivated by the cross-quadratic structure of the equilibrium equation
\eqref{utility-mf-qbsde}, we first
consider the following general mean-field BSDE.  Given
$\xi \in L^q(\MF_T,\R)$ and
$f:[0,T] \times \Omega_T \times \R^{d+1} \times \R^{d+1} \times \R
\longrightarrow \R$, let
\be\label{BSDE1}
Y_t=\xi+ \int_t^T f_r(Y_r,Z_r, \E[Y_r], \E[Z_r],
 \E[Z_r] Z_r) dr - \int_t^T Z_r dB_r.
\ee
where the dependence of functions on $\ome$ is omitted as usual.  For the
local well-posedness analysis, we first consider the following frozen
equation, which will be used in the fixed-point arguments below:
\be\label{BSDE2}
Y_t= \xi + \int_t^T g_r( Y_r, Z_r, a_r  Z_r) dr- \int_t^T Z_r dB_r.
\ee
For the well-posedness of \eqref{BSDE2} and a priori estimate, we assume that 
\begin{assumption}\label{assu:H0}
\textbf{(i)} $\xi \in \ML^q$ with some $q \ge 2$  and $g:[0,T] \times \Omega \times \R \times \R^d \times \R \rightarrow \R$ is progressively measurable with 
$$
\|g(0)\|_{\M^q}^q:=\E  [ \int_0^T |g_r(0,0,0) | dr]^q  < \infty. 
$$
\textbf{(ii)} For any $t \in [0,T] , $ $g_t(\cdot, \cdot, \cdot)$ is uniformly Lipschitz in the sense that for any $(y,z,\MZ),(y',z',\MZ')\in  \R^{d+2} $, for some constant $K>0,$
\begin{equation*}
|g_t(y,z,  \MZ)-g_t(y', z', \MZ')| \le K  ( |y-y'|+ |z-z'|+|\MZ-\MZ'|   ).
\end{equation*}
\end{assumption}

\begin{lem}\label{a-bsde}
Suppose that Assumption~\ref{assu:H0} holds for some $q\ge2$ and that
$a\in L^2([0,T],\R^d)$ is deterministic. Then there exists a unique
solution $(Y,Z)\in\MS^q([0,T])\times\MH^q([0,T],\R^d)$ to
\eqref{BSDE2}. Moreover,
\be\label{a-bsde-est}
\|Y\|_{\MS^q,[0,T] }^q + \| Z \|_{\MH^q,[0,T]  }^q \le C (\| \xi \|_{\ML^q}^q + \|g(0) \|_{\M^q,[0,T]}^q),
\ee
where $C$ depends only on $K,q,T$ and $\|a\|_{L^2([0,T])}$.
	
\end{lem}

\begin{proof}
Firstly, we solve the equation in small time interval. For any $(u,v) \in \MS^q([T_0,T]) \times \MH^q([T_0,T], \R^d) $ with $T_0<T$ to be determined, let $(y,z) \in \MS^2([T_0,T]) \times \MH^2([T_0,T], \R^d)  $ satisfies BSDE with fixed driver
\be\label{BSDE3}
y_t= \xi + \int_t^T g_r( u_r, v_r, a_r  v_r) dr- \int_t^T z_r dB_r,\ \ \  t \in [T_0,T].
\ee
Indeed, note that by Assumption~\ref{assu:H0}(ii) and H\"older's inequality,
\be\label{ineq-1}
\begin{split}
\E[\int_t^T|g_r(u_r,v_r, a_r v_r)|dr]^2 \lesssim_{K,T} & \|g(0)\|_{\M^q}^2 +\|u\|_{\MS^2}^2 + \|v\|_{\MH^2}^2   + \E[\int_t^T |a|^2_r dr \int_t^T |v|^2_r dr  ] 	\\
 \lesssim_{K,T} & \|g(0)\|_{\M^q}^2+ \|u\|_{\MS^2}^2 + \|v\|_{\MH^2}^2   + \|a\|_{L^2}^2 \|v\|_{\MH^2}^2,
\end{split}
\ee
where we apply the fact that $a$ is deterministic in the second inequality. Then according to classical BSDE theory (see e.g. \cite{PP90}), there exists a unique solution $(y,z) \in \MS^2([T_0,T]) \times \MH^2([T_0,T], \R^d)  $ to \eqref{BSDE3}. Now we show that indeed we have $(y,z) \in \MS^q([T_0,T]) \times \MH^q([T_0,T], \R^d).$ Note that $y_t= \E_t[\xi + \int_t^Tg_r(u_r, v_r, a_r v_r)dr],$ and then 
$$
|y_t| \le \E_t[|\xi|+ \int_{T_0}^T |g_r(u_r, v_r, a_r v_r)| dr]=:M_t.
$$
It follows by Doob's maximal inequality that 
\be\label{a-bd-y1}
\|y \|_{\MS^q}^q \le \|M\|_{\MS^q}^q \le C_q \left[\E|\xi|^q + \E[\int_{T_0}^T |g_r(u_r,v_r,a_r v_r)| dr]^q \right].
\ee
Via a similar argument as \eqref{ineq-1}, we have
\be\label{a-bd-g}
\begin{split}
&\E[\int_{T_0}^T |g_r(u_r,v_r,a_r v_r)| dr]^q \\
&\ \ \ \lesssim_{K,q} \|g(0)\|_{\M^q}^q + (T-T_0)^q\|u\|_{\MS^q}^q + (T-T_0)^{\frac{q}{2}} \|v\|_{\MH^q}^q + [\int_{T_0}^T |a_r|^2 dr]^{\frac{q}{2}} \E[ \int_{T_0}^T |v_r|^2 dr]^{\frac{q}{2}}\\
&\ \ \ \lesssim_{K,q} \|g(0)\|_{\M^q}^q + (T-T_0)^q\|u\|_{\MS^q}^q + (T-T_0)^{\frac{q}{2}} \|v\|_{\MH^q}^q + \| a \|_{L^2([T_0,T])}^q  \|v\|_{\MH^q}^q .
\end{split}
\ee
Then in view of \eqref{a-bd-y1} and \eqref{a-bd-g}, we have
\be\label{a-bd-y}
\| y \|_{\MS^q} \lesssim_{K,q} \| \xi \|_{\ML^q} + \|g(0)\|_{\M^q} + (T-T_0) \|u\|_{\MS^q} + \big((T-T_0)^{\frac{1}{2}}+\| a \|_{L^2([T_0,T])}\big) \|v\|_{\MH^q}.
\ee
Next we show $z \in \MH^q.$ According to the Burkholder-Davis-Gundy inequality and BSDE \eqref{BSDE3}, 
\be\label{a-bd-z}
\begin{split}
\E[\int_{T_0}^T |z_r|^2 dr]^{\frac{q}{2}} \lesssim_q & \E|\int_{T_0}^T z_r dB_r|^q \\
\lesssim_q & \| \xi\|^q_{\ML^q} + \E[\int_{T_0}^T |g_r(u_r, v_r, a_r v_r) | dr ]^q + \|y_{T_0}\|_{\ML^q}^q.
\end{split}
\ee
Then in view of \eqref{a-bd-g}, \eqref{a-bd-y} and \eqref{a-bd-z}, $z \in \MH^q.$
Consider the mapping 
\begin{equation*}
\begin{array}{lll}
	\Psi: & \MS^q([T_0,T]) \times \MH^q([T_0,T], \R^d) \longrightarrow & \MS^q([T_0,T]) \times \MH^q([T_0,T],R^d)	\\
	& (u,v) & (y,z).
\end{array}
\end{equation*}
We claim that $\Psi$ is a contraction on every interval on which both its length and the $L^2$ moment of $a$ are sufficiently small. Indeed, for any $(u,v),(\bu,\bv) \in \MS^q \times \MH^q$
 let $(y,z)$ and $(\by,\bz)$ be the corresponding solutions to \eqref{BSDE3}. Denote by $\delta (u,v,y,z):=(u-\bu,v-\bv,y-\by,z-\bz).$ Then $(\delta y, \delta z)$ is a solution to 
\begin{equation*}
\delta y_t= \int_t^T [g_r(u_r, v_r, a_r v_r)- g_r(\bu_r, \bv_r, a_r \bv_r)] dr - \int_t^T \delta z_r dB_r.
\end{equation*}
Then according to \eqref{a-bd-y1} and \eqref{a-bd-z} with $\xi=0$ and $g(0)=0$, we have 
\begin{equation}\label{est:local-a-contract}
\begin{split}
\|\delta y\|_{\MS^q}+\|\delta z\|_{\MH^q}
& \lesssim_{q,K} (T-T_0) \|\delta u\|_{\MS^q}\\
&\quad+\big((T-T_0)^{\frac{1}{2}}
 +\|a\|_{L^2([T_0,T])}\big)\|\delta v\|_{\MH^q}.
\end{split}
\end{equation}
Let
$$
 \Lambda_a(t):=t+\int_0^t|a_r|^2dr,
$$
which is a continuous increasing function with
$\Lambda_a(T)=T+\|a\|_{L^2}^2<\infty$. Note that for any 
$\varepsilon >0$, there is a finite partition
$0=t_0<\cdots<t_N=T$ such that
\[
 |I_i|+\|a\|_{L^2(I_i)}^2\le\varepsilon,
 \qquad I_i=[t_i,t_{i+1}],
\]
with $N$ depending only on $(\varepsilon, T,\|a\|_{L^2})$. Then by choosing $\varepsilon$ small enough, depending only on $(q,K,T,\|a\|_{L^2} )$,  we see from  \eqref{est:local-a-contract} that $\Psi$ a strict contraction on every $I_i$.
Solving backwardly and telescopically over this partition proves the global existence
and uniqueness.

 We finally prove the global estimate \eqref{a-bsde-est}. In view of \eqref{a-bd-y1}, \eqref{a-bd-z} and \eqref{a-bd-g}, on any interval $[s,t]\subseteq[0,T]$, with terminal value $Y_t$, we have
\begin{equation*}
\begin{split}
\|Y\|_{\MS^q,[s,t]} + \| Z \|_{\MH^q,[s,t]}
\lesssim_{K,q}{}& \|Y_t\|_{\ML^q} + \|g(0)\|_{\M^q,[s,t]}
+(t-s)\|Y\|_{\MS^q,[s,t]}\\
&+\big((t-s)^{\frac12}+\|a\|_{L^2([s,t])}\big)
\|Z\|_{\MH^q,[s,t]}.
\end{split}
\end{equation*}
Consequently, on every interval of the above partition,
\be\label{a-bsde-est-local}
\|Y\|_{\MS^q,[s,t]}^q+\|Z\|_{\MH^q,[s,t]}^q
\le C_{K,q}\left(\E|Y_t|^q+\|g(0)\|_{\M^q,[s,t]}^q\right).
\ee

Applying \eqref{a-bsde-est-local} successively
backwardly on the intervals $I_i=[t_i,t_{i+1}]$, we obtain
\[
\sum_{i=0}^{N-1}\left(
\|Y\|_{\MS^q,I_i}^q+\|Z\|_{\MH^q,I_i}^q\right)
\le C_{K,q,N}\left(
\|\xi\|_{\ML^q}^q+
\sum_{i=0}^{N-1}\|g(0)\|_{\M^q,I_i}^q\right).
\]
Note that 
$
\sum_{i=0}^{N-1}\|g(0)\|_{\M^q,I_i}^q
\le \|g(0)\|_{\M^q,[0,T]}^q,
$
and
\[
\|Y\|_{\MS^q,[0,T]}^q
\le\sum_{i=0}^{N-1}\|Y\|_{\MS^q,I_i}^q,
\qquad
\|Z\|_{\MH^q,[0,T]}^q
\le N^{\frac q2-1}\sum_{i=0}^{N-1}\|Z\|_{\MH^q,I_i}^q,
\]
we obtain \eqref{a-bsde-est}. 


\end{proof}

\begin{rem}
	The deterministic nature of $a$ is essential in the preceding localization argument. If $a$ is stochastic, one generally needs additional BMO or
reverse-H\"older bound and lose integrability; see
\cite[Theorem 10]{BC08}.

	
\end{rem}

\begin{coro}\label{BSDE-line}

Suppose $\eta\in\ML^q$, $a\in L^2([0,T],\R^d)$, and
$b\in\M^q([0,T])$ for some $q\ge2$. Let $\alpha$ and $\gamma$ be
real-valued and $\beta$ be $\R^d$-valued bounded progressively measurable
processes, with $|\alpha|\vee|\beta|\vee|\gamma|<K$. Then there exists a
unique solution $(y,z) \in \MS^q([0,T]) \times \MH^q([0,T],\R^d)$ to the linear BSDE with integrable coefficients
\begin{equation*}
y_t= \eta + \int_t^T (b_r+ \alpha_r y_r + \beta_r z_r + \gam_r a_r z_r)dr - \int_t^T z_r dB_r.
\end{equation*}
Moreover, we have for any $t\in [0,T],$
\be\label{est-line}
\begin{split}
	\E|y_t|^q \le C_{T,K,q} \left[ \E|\eta|^q+ \|b\|_{\M^q,[t,T]}^q  \right]e^{2q \int_t^T(1+K^2+K^2 |a_r|^2)dr}.
\end{split}
\ee

\end{coro}

\begin{proof}
According to Lemma \ref{a-bsde}, there exists a unique solution $(y,z) \in \MS^q \times \MH^q.$ We only need to show \eqref{est-line}. Without loss of generality, we claim that we only need to show for any $b \in \MM^q,$
\be \label{simp}
\begin{split}
	\E|y_t|^q \le C_{T,K,q} \left[ \E|\eta|^q+ \|b\|_{\MM^q,[t,T]}^q  \right]e^{2q \int_t^T(1+K^2+K^2 |a_r|^2)dr}.
\end{split}
\ee 
Indeed, let $\by_u:= y_u + \int_t^u b_r dr,$ $u\ge t,$ and then $(\by,z)$ solve the linear BSDE
$$
\by_u = (\eta+\int_t^T b_r dr) + \int_u^T \left[\alpha_r(-\int_t^r b_r dr)+ \alpha_r \by_r + \beta_r z_r + \gam_r a_r z_r \right] dr- \int_u^T z_r dB_r, \ u\in [t,T].
$$
Now consider $(\bar{\eta},\bar{b}_r):=(\eta+\int_t^T b_r dr,-\alpha_r \int_t^r b_r dr)$ and we have 
$$
\|\bar{\eta}\|_{\ML^q}^q \le C_q(\| {\eta}\|_{\ML^q}^q + \|b \|_{\M^q,[t,T]}^q), \ \ \|\bar{b} \|_{\MM^q,[t,T]}^q \le K^q T \| b\|_{\M^q,[t,T]}^q,
$$
which proves our claim.

To show \eqref{simp}, apply It\^o's formula to $|y.|^q$ on $[t,T],$ and we have
\begin{equation*}
\begin{split}
|y_t|^q = &\ |\eta |^q + q \int_t^T |y|^{q-1} \text{sign}(y) (b_r + \alpha_r y_r + \beta_r z_r + \gam_r a_r z_r) dr\\
& - \frac12 q(q-1) \int_t^T |y|^{q-2}	|z|^2 dr- q \int_t^T |y|^{q-1} \text{sign}(y) z dB_r.
\end{split}
\end{equation*}
Let $M_t:=q \int_0^t |y |^{q-1} \text{sign}(y) z dB_r$ be the martingale on $[0,T].$ It follows that
\be\label{ito-0}
\begin{split}
&|y_t|^q +\frac12 q(q-1) \int_t^T |y_r|^{q-2}	|z_r|^2 dr+  (M_T-M_t)\\
&\ \ \   \le |\eta |^q  + q\int_t^T |y_r|^{q-1} (|b_r|  + K|y_r| +K |z_r| + K |a_r z_r|) dr .
\end{split}
\ee
Note that by Young's inequality and Cauchy's inequality for products, we have 
\be\label{pd-0}
\begin{split}
	&  |y_r|^{q-1} ( |b_r |  + K|y_r| +K |z_r| + K |a_r z_r|)  \\
	&\ \ \ \le \frac{q-1}{q}|y_r|^q + \frac{1}{q}  |b_r |^q + K |y_r|^q + K^2 |y_r|^q+ \frac14 |y_r|^{q-2} |z_r|^2\\
	&\ \ \ \ \ \ \  + K^2|y_r|^q |a_r |^2 + \frac14 |y_r|^{q-2} |z_r|^2\\
	&\ \ \ \le \frac{1}{q} |b_r |^q +(\frac{q-1}{q}+K+K^2 + K^2 |a_r |^2) |y_r|^q  + \frac12 |y_r|^{q-2} |z_r|^2 
	\end{split}
\ee
In view of \eqref{ito-0} and \eqref{pd-0}, we have 
\begin{equation*}
|y_t|^q +  (M_T-M_t) \le |\eta |^q+ 2q\int_t^T (1+K^2 + K^2 |a_r|^2 ) |y_r|^q dr + \int_t^T |b_r|^q dr.
\end{equation*}
It follows by taking the expectation on both sides of the above inequality and Gronwall's lemma that 
\begin{equation*}
\E|y_t|^q \le \left[\E|\eta|^q + \E[\int_t^T |b_r |^q dr ]\right]  e^{2q \int_t^T(1+K^2+K^2 |a_r|^2)dr}.
\end{equation*}

\end{proof}

%

	
To apply the Malliavin calculus argument for a non-standard and non-smooth BSDE like \eqref{BSDE2}, we need the following stability result.

\begin{lem}\label{stable1}
Suppose that $(\xi^i,g^i)$ satisfies Assumption~\ref{assu:H0} for some
$q\ge2$, and $a^i\in L^2([0,T],\R^d)$ is deterministic, $i=1,2$.
Let $(y^i,z^i)$ be the solution to \eqref{BSDE2} driven by
$(\xi^i,g^i,a^i)$ respectively. Then
\be\label{a-bsde-stab1}
\begin{split}
& \|y^1- y^2\|_{\MS^q,[0,T] } + \| z^1 -z^2 \|_{\MH^q,[0,T]  } \\
& \ \ \ \le C (\| \xi^1- \xi^2 \|_{\ML^q}  + \|g^1(\Theta^2)-g^2(\Theta^2)  \|_{\M^q,[0,T]} + \|  a^1- a^2\|_{L^2([0,T])} ),
\end{split}
\ee
where $ \Theta^2_r:= ( y^2_r, z^2_r, a^2_r z^2_r)$ and $C$ depends only on
$K,q,T,\|a^2\|_{L^2([0,T])}$ and $\|z^1\|_{\MH^q([0,T])}$.

\end{lem}

\begin{proof}
	
Let $\delta (y,z,a ):=(y^1-y^2, z^1-z^2, a^1-a^2 ).$
It follows that $(\delta y, \delta z)$ satisfies 
\begin{equation*}
\delta y_t = \delta \xi + \int_t^T[\alpha_r \delta y_r + \beta_r \delta z_r + \gamma_r(\delta a_r z_r^1 + a^2_r \delta z_r ) + \delta g_r] dr - \int_t^T \delta z_r dB_r,
\end{equation*}
where 
\be\nonumber
\begin{split}
\alpha_r := & \frac{g^1_r(y^1_r, z^1_r, a^1_r z^1_r)- g^1_r(y^2_r, z^1_r, a^1_r z^1_r )}{\delta y_r}1_{ [|\delta y_r|>0 ]},\\
\beta_r := & \frac{g^1_r(y^2_r, z^1_r, a^1_r z^1_r)- g^1_r(y^2_r, z^2_r, a^1_r z^1_r)}{|\delta z_r|^2}
\delta z_r\,1_{ [|\delta z_r|>0 ]},\\
\gam_r := & \frac{g^1_r(y^2_r, z^2_r, a^1_r z^1_r)- g^1_r(y^2_r, z^2_r, a^2_r z^2_r)}{a^1_r z^1_r- a^2_r z^2_r }1_{ [|a^1_r z^1_r- a^2_r z^2_r|>0 ]}, \\ 
\delta g_r := &  g^1_r(y^2_r, z^2_r, a^2_r z^2_r) - g^2_r(y^2_r, z^2_r, a^2_r z^2_r).
\end{split}
\ee

The local estimate used in the proof of Lemma~\ref{a-bsde} applies to this
linear BSDE on every interval $I=[s,t]$ satisfying
$|I|+\|a^2\|_{L^2(I)}^2\le\varepsilon(q,K)$, and gives
\begin{equation*}
\begin{split}
\|\delta y\|_{\MS^q,I}^q+\|\delta z\|_{\MH^q,I}^q
\le C_{K,q}\big(&\E|\delta y_t|^q+\|\delta g\|_{\M^q,I}^q
+\|\delta a\,z^1\|_{\M^q,I}^q\big).
\end{split}
\end{equation*}
Moreover, H\"older's inequality gives
\[
\|\delta a\,z^1\|_{\M^q,I}^q
\le \|\delta a\|_{L^2(I)}^q\|z^1\|_{\MH^q(I)}^q.
\]
It follows that, on every interval of an energy partition associated with
$a^2$,
\[
\|\delta y\|_{\MS^q,I}^q+\|\delta z\|_{\MH^q,I}^q
\le C_{K,q}\left(\E|\delta y_t|^q
+\|\delta g\|_{\M^q,I}^q
+\|z^1\|_{\MH^q([0,T])}^q\|\delta a\|_{L^2(I)}^q\right).
\]
In view of the fact $\sum_i\|\delta a\|_{L^2(I_i)}^q\le
\|\delta a\|_{L^2([0,T])}^q$, the global estimate \eqref{a-bsde-stab1} is then obtained by a telescopic
argument as in the proof of Lemma~\ref{a-bsde}.


\end{proof}

Now let us study the Malliavin differentiability of BSDE \eqref{BSDE2}. 
In the following we write $\Theta_t= (Y_t, Z_t, a_t Z_t)$ and $\nabla  g= (\partial_y, \partial_z, \partial_{\MZ} )g$.


\begin{assumption}\label{assu:H1}
Suppose that Assumption~\ref{assu:H0} holds for some $q\ge2$. In addition, for any $t \in [0,T],$ $g_t(\cdot,\cdot,\cdot)$ is continuously differentiable on $\R^{d+2}.$
Moreover, $\xi \in \D^{1,2}$ and
$g_\cdot(y,z,\MZ) \in \bL^{1,2}$ for every $(y,z,\MZ) \in \R^{d+2}$.
The random field $D_ug_r(\cdot)$ admits a jointly measurable version such
that, for $du\otimes dr\otimes dP$-a.e. $(u,r,\omega)$,
$
 D_ug_r(\cdot)\in\MC_b(\R^{d+2},\R^d).
$
Finally, write $g_r(0):=g_r(0,0,0)$ and assume that for some $K_0>0$,
\begin{equation*}
\|\xi\|^q_{\ML^q} +  \| g(0) \|_{\M^q}^q
+ \E \left[ \int_0^T |D_{u} \xi |^q du
+  \int_0^T \int_0^T\|D_u g_r(\cdot)\|_{\MC_b}^q drdu\right] <K_0.
\end{equation*}

\end{assumption}

\begin{prop}\label{m-d}
Suppose that Assumption~\ref{assu:H1} holds with $q\ge2$ and that
$a\in L^2([s,T],\R^d)$ is deterministic with some $s\in[0,T)$. Let
$(Y,Z) \in \MS^q([s,T]) \times \MH^q([s,T])$ be the unique solution to
\eqref{BSDE2} given by Lemma~\ref{a-bsde}.
Then
$(Y,Z) \in \bL^{1,2}([s,T],\R^{d+1})$, and there is a jointly measurable
version of $(D_uY_t,D_uZ_t)_{0\le u\le T,\,s\le t\le T}$ satisfying
$D_uY_t=D_uZ_t=0$ for $s\le t<u\le T$ and, for
$0\le u\le t$ and $s\le t\le T$,
\begin{equation*}
D_u Y_t = D_u \xi  + \int_t^T [D_u g_r + \nabla g_r \,   D_u\Theta_r    ](\Theta_r )dr- \int_t^T D_u Z_r dB_r.
\end{equation*}
Moreover, for every $t\in[s,T]$ and a.e. $u\in[0,t]$,
\be\label{est:Duyt}
 \E|D_uY_t|^q
 \le\left[\E|D_u\xi|^q
 +\int_t^T\E\|D_ug_r(\cdot)\|_{\MC_b}^qdr\right]
 \exp\left\{2q\int_t^T
       (1+K^2+K^2|a_r|^2)dr\right\}.
\ee
Finally, $(D_tY_t)_{t\in[s,T]}$ is a version of
$(Z_t)_{t\in[s,T]}$.  Consequently, $Z\in\MM^q([s,T],\R^d)$ and
\be\label{m-d-zq}
\begin{split}
 \int_s^T\E|Z_t|^qdt
 \le \left[\int_s^T\E|D_t\xi|^qdt
 +\int_s^T\int_t^T
 \E\|D_tg_r(\cdot)\|_{\MC_b}^qdrdt\right]
 e^{ 2q\int_s^T(1+K^2+K^2|a_r|^2)dr }.
\end{split}
\ee

\end{prop}

\begin{proof}
Without loss of generality, we assume $s=0$. Let $a^{(m)}$ be obtained by truncating every component of $a$ to
$[-m,m]$, $m\in\N$.  Set $(y^{(0,m)},z^{(0,m)})=(0,0)$ and recursively let
$(y^{(n+1,m)},z^{(n+1,m)})\in\MS^q\times\MH^q$ be the solution to
\begin{equation*}
y^{(n+1,m)}_t = \xi + \int_t^T g	_r(y^{(n,m)}_r, z^{(n,m)}_r,  a_r^{(m)} z^{(n,m)}_r) dr - \int_t^T z^{(n+1,m)}_r dB_r.
\end{equation*}
Note that $\|a^{(m)} \|_{L^2} \le \|a\|_{L^2},$ $\forall m,$ in view of the proof of Lemma \ref{a-bsde},
$(y^{(n,m)},z^{(n,m)})$ converges to $(y^{(m)},z^{(m)})$ as $n$ goes to infinity in the $\MS^q({[0, T]}) \times \MH^q({[0, T]},\R^d)$ norm, where $(y^{(m)},z^{(m)})$ is the solution to 
\begin{equation*}
y^{( m)}_t = \xi + \int_t^T g	_r(y^{( m)}_r, z^{( m)}_r,  a_r^{(m)} z^{( m)}_r) dr - \int_t^T z^{( m)}_r dB_r.
\end{equation*}
On the other hand, according to Lemma \ref{stable1}, 
$$
\|Y-y^{(m)}\|_{\MS^q,[0,T]} + \|Z-z^{(m)} \|_{\MH^q,[0,T]} \le C \|a-a^{(m)}\|_{L^2},
$$
for a constant $C>0$ depending on $\|a\|_{L^2}$, which implies that $(y^{(n,m)},z^{(n,m)})$ converges to $(Y,Z)$ in the $\MS^q({[0 , T]}) \times \MH^q({[0, T]},\R^d)$ norm as $m$ and $n $ go to infinity. It follows that, by the closedness of the Malliavin differential operator $D$, we only need to show $(y^{(n,m)},z^{(n,m)})$ is a Cauchy sequence in $ \bL^{1,2}({[0, T]}) \times \bL^{1,2}({[0, T]},\R^d)$ to conclude $(Y,Z) \in \bL^{1,2}({[0, T]}) \times \bL^{1,2}({[0, T]},\R^d).$

Applying the random-field chain rule \eqref{random-field-chain-rule} for the driver, and \cite[Lemma 5.1]{KPQ97} for the stochastic
integral, we see that, via an induction argument, 
$(y^{(n+1,m)},z^{(n+1,m)})\in\bL^{1,2}([0,T])\times
\bL^{1,2}([0,T],\R^d)$ and, for $0\le u\le t\le T$,
\be\label{nm-deri}
D_u y^{(n+1,m)}_t= D_u \xi + \int_t^T [D_u g_r(\tTheta^{(n,m)}_r ) + \nabla g_r(\tTheta^{(n,m)}_r) D_u \tTheta^{(n,m)}_r ] dr - \int_t^T D_u z^{(n+1,m)}_r dB_r,
\ee
where $\tTheta^{(n,m)} = (y^{(n,m)},z^{(n,m)},
a^{(m)}z^{(n,m)})$.
We claim that $(y^{(n,m)},z^{(n,m)})$ converges under
$\bL^{1,2}$-norm to $(y^\cdot,z^\cdot)$, where
$(y^u_t,z^u_t;\ u\le t\le T)\in\MS^q\times\MH^q$ is the unique solution
to the linear BSDE
\be\label{M-deri}
  y_t^u = D_u \xi  + \int_t^T ( D_u g_r(\Theta_r) + \nabla g_r(\Theta_r) \Theta^u_r  ) dr- \int_t^T  z^u_r dB_r,
\ee
with $\Theta^u=(y^u,z^u,a z^u)$,
which is nothing but $d$ copies of the scalar linear BSDE and thus well-posed by Corollary~\ref{BSDE-line}.
Indeed, note that
\[
 \int_0^T\E\left(\int_u^T|D_ug_r(\Theta_r)|dr\right)^qdu
 \le T^{q-1}\E\int_0^T\int_u^T
 \|D_ug_r(\cdot)\|_{\MC_b}^qdrdu<\infty.
\]
Thus the inhomogeneous term in \eqref{M-deri} belongs to $\M^q$ and Corollary~\ref{BSDE-line} implies $(y^u, z^u) \in \MS^q \times \MH^q$. 
According to \eqref{nm-deri}, \eqref{M-deri} and the standard estimate for
BSDEs (or Lemma \ref{a-bsde}), we have
\be\label{decom-d}
\begin{split}	
	\|D_u y^{(n+1,m)} - y^u \|_{\MS^2,[u,T]}^2 +  \|D_u z^{(n+1,m)} - z^u \|_{\MH^2,[u,T]}^2 \le C \E\left[ \int_u^T  (d^1_r  +  d^2_r  + d^3_r)  dr \right]^2 ,
\end{split}
\ee
where 
\begin{equation*}
\begin{split}
	&d^1_r:=|D_u g_r(\tTheta^{(n,m)}_r)-  D_u g_r(\Theta_r)|,\\
	&d^2_r:=|\partial_{\MZ}g_r(\tTheta^{(n,m)}_r) a^{(m)}_r D_u z^{(n,m)}_r - \partial_{\MZ}g_r(\Theta_r) a_r z^u_r|,\\
	&d^3_r:= |\partial_y g_r(\tTheta^{(n,m)}_r ) D_u y^{(n,m)}_r - \partial_y g_r(\Theta_r )  y^u_r|+|\partial_z g_r(\tTheta^{(n,m)}_r ) D_u z^{(n,m)}_r - \partial_z g_r(\Theta_r )  z^u_r|,
\end{split}	
\end{equation*}
We estimate the
three terms on the right hand side of \eqref{decom-d} after integration over $u$. First, note that $\tTheta^{(n,m)}\to\Theta$ in probability. By the bounded convergence theorem, we have 
\begin{equation}\label{R1}
\begin{split}
 R^1_{n,m}
 :=\int_0^T\E\left(\int_u^Td^1_rdr\right)^2du \le T\E\int_0^T\int_u^T
 |D_ug_r(\tTheta_r^{(n,m)})-D_ug_r(\Theta_r)|^2drdu \rightarrow 0.
\end{split}
\end{equation}
For $\ell\in\{y,z,\MZ\}$, write
$
 \Delta^{n,m}\partial_\ell g_r
 :=\partial_\ell g_r(\tTheta_r^{(n,m)})
   -\partial_\ell g_r(\Theta_r).
$
For the second term, by the fact that $a$ is determinisitic,
\begin{equation}\label{R2}
\begin{split}
 \int_0^T\E\left(\int_u^Td^2_rdr\right)^2du 
 \lesssim_K \|a\|_{L^2}^2
 \int_0^T\|D_uz^{(n,m)}-z^u\|_{\MH^2([u,T])}^2du   +R^2_{n,m},
\end{split}
\end{equation}
Here
$
R^2_{n,m}:=  \|a\|_{L^2}^2
 \E\int_0^T\int_0^r
 |\Delta^{n,m}\partial_{\MZ}g_r|^2|z_r^u|^2dudr + \|a^{(m)}-a\|_{L^2}^2
 \int_0^T\|z^u\|_{\MH^2([u,T])}^2du,
$
which converges to zero as before.
The same
argument yields
\be\label{d3}
\begin{split}
 &\int_0^T\E\left(\int_u^Td^3_rdr\right)^2du\\
 &\quad\le C_K(T+T^2)\int_0^T\left(
 \|D_uy^{(n,m)}-y^u\|_{\MS^2([u,T])}^2
 +\|D_uz^{(n,m)}-z^u\|_{\MH^2([u,T])}^2\right)du
 +R^3_{n,m},
\end{split}
\ee
where
\[
 R^3_{n,m}:=C_KT\E\int_0^T\int_0^r
 \left(|\Delta^{n,m}\partial_yg_r|^2|y_r^u|^2
 +|\Delta^{n,m}\partial_zg_r|^2|z_r^u|^2\right)dudr \rightarrow 0.
\]
Set
\[
A_{n,m}:=\int_0^T\left(
\|D_u y^{(n,m)}-y^u\|_{\MS^2([u,T])}^2
+\|D_u z^{(n,m)}-z^u\|_{\MH^2([u,T])}^2\right)du.
\]
Combining \eqref{decom-d}--\eqref{d3}, we obtain
\[
 A_{n+1,m}\le \rho A_{n,m}+R_{n,m},
\]
where $\rho$ and $R_{n,m}$ are constants such that 
\[
 \rho\le C_{K,q,T}\bigl(T+\|a\|_{L^2}^2\bigr),\ \ \ \limsup_{m\to\infty}\limsup_{n\to\infty} R_{n,m}=0.
\]
To show the convergence, we need to localize the above estimate. For an interval $I=[r,s]\subset[0,T]$,
let
\[
\begin{split}
 A_{n,m}^{I}:={}&\int_0^s\Big(
 \|D_uy^{(n,m)}-y^u\|_{\MS^2([u\vee r,s])}^2
 +\|D_uz^{(n,m)}-z^u\|_{\MH^2([u\vee r,s])}^2\Big)du,\\
 E_{n,m}^{s}:={}&\int_0^s\E
 \big|D_uy_s^{(n,m)}-y_s^u\big|^2du.
\end{split}
\]
Keeping track of an extra terminal-value in the above estimates and
localizing \eqref{decom-d}--\eqref{d3}, we have
\be\label{eq:local-estA}
 A_{n+1,m}^{I}
 \le C E_{n+1,m}^{s}+\rho_I A_{n,m}^{I}+R_{n,m}^{I},
 \qquad
 \rho_I\le C_{K,q,T}
 \bigl(|I|+\|a\|_{L^2(I)}^2\bigr),
\ee
where the localized remainder $R_{n,m}^{I}$ has the same vanishing
property as $R_{n,m}$. Note that 
$\mu(dt)=(1+|a_t|^2)dt$ is finite and nonatomic. Hence we may choose a partition $0=t_0<\cdots<t_J=T$ such that
$\rho_{[t_j,t_{j+1}]}<1$ for every $j$.
It remains to perform a backward induction to complete the convergence. Set $I_j=[t_j,t_{j+1}]$ and
\[
 L_j:=\limsup_{m\to\infty}\limsup_{n\to\infty}A_{n,m}^{I_j},
 \qquad
 F_j:=\limsup_{m\to\infty}\limsup_{n\to\infty}E_{n,m}^{t_j}.
\]
Then by taking the limit in \eqref{eq:local-estA}, we have
\[
 (1-\rho_{I_j})L_j\le C F_{j+1}.
\]
Note that for any $m,n\ge1$, on interval $I_{J-1}$, we have
$E_{n,m}^{T}=0$ and thus $L_{J-1}=F_J=0$. Moreover, by definition we have
\[
 E_{n,m}^{t_j}\le A_{n,m}^{I_j},
 \qquad\text{and consequently}\qquad F_j\le L_j.
\]
It follows that $F_{J-1}=0$. Then repeat the same argument above and we obtain $L_j=F_j=0$ for any $j=0,...,J-2,$ which implies the Malliavin derivatives converge on every interval of the
partition. Hence on $[0,T]$, 
\[
 (Y,Z)\in\bL^{1,2}([0,T],\R^{d+1}),
 \qquad (D_uY_t,D_uZ_t)=(y^u_t,z^u_t),\quad u\le t.
\]
for $du\,dt\,dP$-almost every $(u,t,\omega)$.  

By the same argument as \cite[Proposition 5.3]{KPQ97}, we have
$
 Z_u=y_u^u,\  du\otimes dP\text{-a.e.   }
$
Finally, applying It\^o's formula to $|y^u|^q$, and by Young's inequality and Gronwall's lemma, we have that, for
$0\le u\le t$ and $t\in[s,T]$,
\[
 \E|y_t^u|^q
 \le\left[\E|D_u\xi|^q
 +\int_t^T\E\|D_ug_r(\cdot)\|_{\MC_b}^qdr\right]
 \exp\left\{2q\int_t^T(1+K^2+K^2|a_r|^2)dr\right\}.
\]
Moreover, taking $t=u$, integrating over $u\in[s,T]$, and applying $Z_u=D_uY_u$, we have
\eqref{m-d-zq} and complete the proof.

\end{proof}

\begin{rem}\label{rem:malliavin-linear-growth}
In view of the classical Lipschitz condition introduced in \cite[Proposition~5.3]{KPQ97}, the uniform-in-state bound on $D_ug_r(\cdot)$ in
Assumption~\ref{assu:H1} can be replace by a similar Lipschitz kind of conditions to conclude
Proposition~\ref{m-d}: there exists
$K>0$ such that
\[
 |D_ug_r(x)-D_ug_r(x')|\le K|x-x'|,
 \qquad |D_ug_r(0)|\le R_{u,r},
\]
where $R_{u,r}$ is a jointly measurable nonnegative process with enough integrability.
However, the uniform bound in the Assumption~\ref{assu:H1} is designed to prove the
global result of Section~\ref{sec:global}, where it provides the
solution-independent uniform bound on Malliavin derivatives
$
 |D_u\xi|+\int_u^T\|D_uF_r(\cdot)\|_{\MC_b}\,dr
$
which is vital for the extension of local solutions to global ones.
\end{rem}

\subsection{The case of drivers Lipschitz in \(\E[Z] Z\)}

We now formulate the local well-posedness result directly for
law-dependent coefficients.  This formulation contains
\eqref{BSDE1} as a special case. Consider
\be\label{bsde-law}
y_t= \xi + \int_t^T F_r(y_r,z_r,\ML(y_r,z_r), G_r(\ML({z_r}))z_r ) dr - \int_t^T z_r dB_r,
\ee
where $F: [0,T] \times \Omega_T \times \R \times \R^d \times \op_1(\R^{d+1}) \times \R \rightarrow \R$ is progressively measurable and $G:[0,T] \times \op_1(\R^{d}) \rightarrow \R^d$ is deterministic and Borel measurable.
\begin{assumption}\label{assu:H2prime}
There exists $q>2$ such that the following hold.
(i) $F$ is uniformly Lipschitz in the sense that for any $t \in [0,T],$ $(y,z, \MZ),(y',z',\MZ')\in  \R^{d+2} $ and $\mu,\mu' \in \op_1(\R^{d+1}),$
$$
|F_t(y,z,\mu, \MZ)- F_t(y', z', \mu',\MZ')| \le K  ( |y-y'|+ |z-z'|+ W_1(\mu,\mu') +|\MZ-\MZ'|   ), \ \text{for }K>0;
$$
$G_t$ is uniformly Lipschitz on $\op_1(\R^d)$ and
$G_t(\delta_0)\in L^2([0,T],\R^d)$, i.e. for any
$\rho,\rho'\in\op_1(\R^d)$,
$$
|G_t(\rho)-G_t(\rho')|\le K W_1(\rho, \rho'),\ \ \
G_2:=\int_0^T|G_t(\delta_0)|^2dt<\infty.
$$
(ii) $\xi\in\ML^q\cap\D^{1,2}$ and, for every fixed
$(y,z,\mu,\MZ)\in\R\times\R^d\times\op_1(\R^{d+1})\times\R$,
$F_\cdot(y,z,\mu,\MZ)\in\bL^{1,2}$.  The random field
$D_uF_r(\cdot)$ admits a jointly measurable version such that
\[
 D_uF_r(\cdot)\in
 \MC_b\big(\R^{d+2}\times\op_1(\R^{d+1}),\R^d\big)
\]
for $du\otimes dr\otimes dP$-a.e. $(u,r,\omega)$, and, for some
$K_0>0$,
\begin{equation}\label{law-assu:K0}
\E [\int_0^T |D_{u} \xi |^q du + \int_0^T \int_0^T\|D_u F_r(\cdot)\|_{\MC_b}^q du dr + \int_0^T |F_r(0)|^q dr ]<K_0,
\end{equation}
where $F_r(0)=F_r(0,0,\delta_0,0)$.  

\end{assumption}

Then we have the local well-posedness of \eqref{bsde-law}.

\begin{thm}\label{local-wp-law}
Suppose that Assumption~\ref{assu:H2prime} holds for $(\xi,F,G)$.
Then there exists $\delta\in(0,T]$, depending only on
$(q,K,K_0,G_2)$, such that, for every $s\in[T-\delta,T)$,
BSDE~\eqref{bsde-law} admits a unique solution
$(Y,Z)\in\MS^q([s,T])\times\MM^q([s,T],\R^d)$.  Moreover, there is a constant
$C=C(q,K,K_0,G_2)>0$ such that for $R:=(eK_0\exp(4qK^2G_2))^{\frac1q}$,
\be\label{law-local-apriori}
 \|Y\|_{\MS^q,[s,T]}
 \le C\left[\|\xi\|_{\ML^q}^q
 +(T-s)^{q-1}(K_0+R^q)\right]^{\frac1q},\ \ \
 \|Z\|_{\MM^q,[s,T]}
 \le R.
\ee

\end{thm}

\begin{proof}
\textbf{Step 1: Construction of the fixed point mapping.}
For $s\in[0,T)$, set $I_s=[s,T]$, $\ell=T-s$ with $\ell$ to be determined, and
$
 \alpha:=\frac12-\frac1q>0.
$
Consider
\begin{equation*}
 \MU_s:=\left\{(u,v)\in\MS^q(I_s)\times\MH^q(I_s,\R^d):
 v\in\MM^q(I_s,\R^d),\quad
 \|v\|_{\MM^q,I_s}\le R\right\},
\end{equation*}
equipped with the $\MS^q\times\MH^q$ metric. It is standard to check the above space is a complete metric space.
For $(u,v)\in\MU_s$, let
\[
 \mu_r:=\ML(u_r,v_r),\qquad \rho_r:=\ML(v_r),\qquad
 a_r:=G_r(\rho_r).
\]
Then by the Lipschitz
property of $G$, we have 
\be\label{law-frozen-coefficient}
 \|a\|_{L^2(I_s)}^2
 \le 2G_2+2K^2\ell^{2\alpha}
       \|v\|_{\MM^q,I_s}^2
 \le 2G_2+2K^2R^2.
\ee
Moreover,
\[
 |F_r(0,0,\mu_r,0)|
 \le |F_r(0)|+K\big(\E|u_r|+\E|v_r|\big),
\]
so $F_\cdot(0,0,\mu_\cdot,0)\in\M^q(I_s)$.  Then by Lemma~\ref{a-bsde}, there exists a unique $(y,z)\in\MS^q(I_s)\times\MH^q(I_s,\R^d)$
solving
\be\label{law-frozen-fixed-bsde}
 y_t=\xi+\int_t^T F_r(y_r,z_r,\mu_r,a_rz_r)dr
       -\int_t^Tz_rdB_r.
\ee
Denote this map by $\Psi(u,v)=(y,z)$.

\textbf{Step 2: Invariance via Malliavin calculus.}
We now prove that $\Psi$ maps $\MU_s$ into itself. Indeed, we only need to show $\|z\|_{\MM^q,I_s}\le R$. Fix $(u,v)\in\MU_s$ and let
\[
 g_r(x):=F_r(x_0,x_1,\mu_r,x_2),
 \qquad x=(x_0,x_1,x_2)\in\R\times\R^d\times\R.
\]
It is standard to check that $g_\cdot(x)\in\bL^{1,2} $, and moreover,
\be\label{eq:unibd-dg}
 D_ug_r(x)=D_uF_r(x_0,x_1,\mu_r,x_2),\qquad
 \|D_ug_r(\cdot)\|_{\MC_b(\R^{d+2})}
 \le\|D_uF_r(\cdot)\|_{\MC_b}.
\ee
Now convolve $g$ only in $x=(y,z,\MZ)$ with a standard mollifier $\phi^\varepsilon$ and
write $g^\varepsilon=g*\phi^\varepsilon$. Then
\be\label{ine:uni-lip-g}
 |g_r^\varepsilon(x)-g_r(x)|\le CK\varepsilon,
 \qquad \|\nabla g_r^\varepsilon\|_\infty\le K,
 \qquad
 D_ug_r^\varepsilon=(D_ug_r)*\phi^\varepsilon.
\ee
Let
$(y^\varepsilon,z^\varepsilon)$ solve \eqref{law-frozen-fixed-bsde}
with $g$ replaced by $g^\varepsilon$. Note that $(\xi, g^{\vep}, a)$ satisfies
Assumption~\ref{assu:H1} with the same bound as $(\xi, g, a)$, and thus by
Proposition~\ref{m-d}, we have 
\be\label{law-frozen-mq-estimate}
\begin{split}
 \|z^\varepsilon\|_{\MM^q,I_s}^q
 \le{}&K_0\exp\left\{2q\int_s^T
       (1+K^2+K^2|a_r|^2)dr\right\}\\
 \le{}&K_0\exp\left\{2q\left[(1+K^2)\ell
       +2K^2G_2+2K^4R^2\ell^{2\alpha}\right]\right\}.
\end{split}
\ee
Choose $\delta_1\in(0,1]$, depending only on
$(q,K,K_0,G_2)$, so that
\be\label{law-invariance-choice}
 2q\left[(1+K^2)\delta_1
       +2K^4R^2\delta_1^{2\alpha}\right]\le\frac12.
\ee
By Lemma~\ref{stable1} with $(g^1,y^1,z^1)=(g,y,z)$ and
$(g^2,y^2,z^2)=(g^\varepsilon,y^\varepsilon,z^\varepsilon)$, and \eqref{ine:uni-lip-g}, we have
\[
 \|y^\varepsilon-y\|_{\MS^q,I_s}
 +\|z^\varepsilon-z\|_{\MH^q,I_s}
 \le C\|g^\varepsilon(\Theta^\varepsilon)
              -g(\Theta^\varepsilon)\|_{\M^q,I_s}
 \le C\varepsilon\ell,
\]
where $\Theta_r^\varepsilon=(y_r^\varepsilon,z_r^\varepsilon,
a_rz_r^\varepsilon)$.
Thus $(y^\varepsilon,z^\varepsilon)\to(y,z)$ in
$\MS^q(I_s)\times\MH^q(I_s)$ as $\varepsilon\downarrow0$.
Passing to an a.e. convergent subsequence and using Fatou's lemma in
\eqref{law-frozen-mq-estimate}, we obtain, whenever $\ell\le\delta_1$,
\[
 \|z\|_{\MM^q,I_s}^q
 \le e^{1/2}K_0\exp(4qK^2G_2)<R^q.
\]
Thus $\Psi(\MU_s)\subset\MU_s$.

\textbf{Step 3: Contraction and existence of solutions.}
We next prove contraction. For $(u^i,v^i)\in\MU_s$, let
$(y^i,z^i)=\Psi(u^i,v^i)$ and $(\Delta u, \Delta v):=(
u^1-u^2, v^1-v^2 )$. Similar notations like $\Delta y, \Delta z$ are self-explained. Let $\mu^i_r=\ML(u^i_r,v^i_r)$ and
$a^i_r=G_r(\ML(v^i_r))$. It follows that
\be\label{law-coupling-bounds}
\begin{split}
 W_1(\mu^1_r,\mu^2_r)
 \le\E|\Delta u_r|+\E|\Delta v_r|,\ \ 
 |a^1_r-a^2_r| \le K\E|\Delta v_r|.
\end{split}
\ee
By taking the differnce of two equations, we have 
that for some progressively measurable processes $\alpha, \beta, \gamma$, all bounded by $K$, 
\be\label{eq:dif-yz}
 \Delta y_t=\int_t^T\bigl(
   \alpha_r\Delta y_r+\beta_r\Delta z_r
   +\gamma_r a^1_r\Delta z_r+b_r\bigr)dr
   -\int_t^T\Delta z_r dB_r,
\ee
where 
$
 b_r:=F_r(y^2_r,z^2_r,\mu^1_r,a^1_rz^2_r)
      -F_r(y^2_r,z^2_r,\mu^2_r,a^2_rz^2_r).
$
Then by Lemma~\ref{a-bsde} we have
\be\label{law-contraction-pre}
 \|\Delta y\|_{\MS^q,I_s}+\|\Delta z\|_{\MH^q,I_s}
 \le C_{q,K,K_0,G_2}\|b\|_{\M^q,I_s} \le C_{q,K,K_0,G_2}\left\|
 W_1(\mu^1,\mu^2)+|a^1-a^2|\,|z^2|
 \right\|_{\M^q,I_s}.
\ee
To estimate the right hand side of the above inequality, by
\eqref{law-coupling-bounds},
\[
 \left\|W_1(\mu^1,\mu^2)\right\|_{\M^q,I_s}
 \le \ell\|\Delta u\|_{\MS^q,I_s}
      +\ell^{1/2}\|\Delta v\|_{\MH^q,I_s}.
\]
Furthermore, by the Cauchy--Schwarz inequality and
$\|z^2\|_{\MH^q,I_s}\le
\ell^\alpha\|z^2\|_{\MM^q,I_s}$, we have
\[
\begin{split}
 \left\||a^1-a^2| \ |z^2|\right\|_{\M^q,I_s}
 \le K\Big(\int_s^T(\E|\Delta v_r|)^2dr\Big)^{\frac12}
          \|z^2\|_{\MH^q,I_s} \le KR\ell^{\alpha}\|\Delta v\|_{\MH^q,I_s}.
\end{split}
\]
Combining the above bounds with \eqref{law-contraction-pre} we obtain
\begin{equation*}
\begin{split}
 \|\Delta y\|_{\MS^q,I_s}+\|\Delta z\|_{\MH^q,I_s}
 \le C_{q,K,K_0,G_2}\big[ \ell\|\Delta u\|_{\MS^q,I_s} +(\ell^{1/2}+KR\ell^\alpha)
       \|\Delta v\|_{\MH^q,I_s}\big].
\end{split}
\end{equation*}
Choose $\delta_2\in(0,\delta_1]$ such that
\[
 C_{q,K,K_0,G_2}\big(\delta_2+\delta_2^{1/2}
                   +KR\delta_2^\alpha\big)\le\frac12.
\]
Then $\Psi$ is a contraction whenever $\ell\le\delta_2$, which implies there exists a fixed point
$(Y,Z)$ in $\MU_s$ solving 
\eqref{bsde-law}.

\textbf{Step 4: Uniqueness and estimate.}
Now we prove uniqueness in the whole space $\MS^q \times \MM^q$, rather than
only inside $\MU_s$. This follows by the fact that any solution in $\MS^q(I_s)\times\MM^q(I_s,\R^d)$ is a fixed point of $\Psi$. Indeed, let $(\widetilde Y,\widetilde Z)\in\MS^q(I_s)\times\MM^q(I_s,\R^d)$ be
any solution and set
\[
 X(t):=\int_t^T\E|\widetilde Z_r|^qdr,
 \qquad \widetilde a_r:=G_r(\ML(\widetilde Z_r)).
\]
We only need to show that $X(s)\le R^q$. Following the same argument as above, now with the two law processes of this
solution frozen, we have for any $t\in [s,T],$
\be\label{law-solutionwise-mq-estimate}
 X(t)\le K_0\exp\left\{2q\left[(1+K^2)(T-t)
       +K^2\int_t^T|\widetilde a_r|^2dr\right]\right\}.
\ee
Letting $h=T-t$, by H\"older's inequality we have 
\be\label{ine:tilde-a}
 \int_t^T|\widetilde a_r|^2dr
 \le 2G_2+2K^2h^{2\alpha}X(t)^{2/q}.
\ee
If $X(s)>R^q$, then there exists $\tau \in (s,T)$ such that $X(\tau)=R^q$. Then by 
\eqref{law-solutionwise-mq-estimate}, \eqref{ine:tilde-a},
\eqref{law-invariance-choice}, we have
\[
 R^q=X(\tau)
 \le e^{1/2}K_0\exp(4qK^2G_2)<R^q,
\]
a contradiction.  Hence our claim holds and the uniqueness follows.

It remains to show the estimate for $Y$. Fixing $t\in[s,T]$, and
$h=T-t$, set
$$
 \mu_r:=\ML(Y_r,Z_r),\ 
 a_r:=G_r(\ML(Z_r)),\  r\in[t,T].
$$
Note that 
\[
 |F_r(0,0,\mu_r,0)|
 \le |F_r(0)|+K\bigl(\E|Y_r|+\E|Z_r|\bigr).
\]
It follows by H\"older's inequality that
\[
\begin{split}
 \|F_\cdot(0,0,\mu_\cdot,0)\|_{\M^q,[t,T]}
 \le h^{1-1/q}K_0^{1/q}
 +Kh\|Y\|_{\MS^q,[t,T]} + Kh^{1-1/q}\|Z\|_{\MM^q,[t,T]}.
\end{split}
\]
Then by applying Lemma~\ref{a-bsde} with $\mu$ and $a$ frozen, 
\[
 \|Y\|_{\MS^q,[t,T]}
 \le C_{q,K,K_0,G_2}\left(
 \|\xi\|_{\ML^q}+h^{1-1/q}K_0^{1/q}
 +Kh\|Y\|_{\MS^q,[t,T]}+Kh^{1-1/q}R\right).
\]
Decreasing $\delta$ once more so that
$C_{q,K,K_0,G_2}K\delta\le\frac12$, we obtain the estimate of $Y$.




\end{proof}

\begin{rem}\label{rem:expectation-specialization}
The expectation-dependent equation \eqref{BSDE1} is a particular case of 
\eqref{bsde-law}. Indeed, for
$\mu\in\op_1(\R\times\R^d)$ and $\rho\in\op_1(\R^d)$, set
\[
\begin{split}
 F_t(y,z,\mu,\MZ)
 :=
 f_t\big(y,z,
          \int_{\R\times\R^d}y'\,\mu(dy',dz'),
          \int_{\R\times\R^d}z'\,\mu(dy',dz'),
          \MZ\big),\ \ 
 G_t(\rho):=\int_{\R^d}z'\,\rho(dz').
\end{split}
\]
Consequently, whenever
$(\xi,F,G)$ satisfies Assumption~\ref{assu:H2prime},
Theorem~\ref{local-wp-law} applies directly to the
expectation-dependent equation.

\end{rem}

\subsection{The case of drivers Lipschitz in $(|\E[Z]|^2, \E[Z]Z)$}

To include the case with driver containing the term $|\E[Z] |^2,$ we further extend the result of Theorem \ref{local-wp-law} by exploiting the additivity structure of equations like \eqref{utility-mf-qbsde}. Now consider
\be\label{bsde-law-lb}
y_t= \xi + \int_t^T H_r(y_r,z_r,\ML(y_r), \ML(z_r), G_r(\ML({z_r}))z_r ) dr - \int_t^T z_r dB_r,
\ee
where $H:[0,T]\times\Omega_T\times\R\times\R^d\times
\op_1(\R)\times\op_1(\R^d)\times\R\to\R$ is progressively
measurable and $G:[0,T]\times\op_1(\R^d)\to\R^d$ is deterministic.

\begin{assumption}\label{assu:H3}
For every $(t,y,z,\nu,\rho,\MZ)\in [0,T]\times\R\times \R^d \times \op_1(\R ) \times \op_1(\R^d) \times \R$, suppose
\be \label{decom}
H_t(y,z, \nu, \rho, \MZ)= F_t(y,z,\nu, \rho,\MZ)+ h_t(\nu,\rho).
\ee
For $\mu\in\op_1(\R^{d+1})$, let $\mu^Y$ and $\mu^Z$ denote its two
marginals and set
\[
 \widehat F_t(y,z,\mu,\MZ)
 :=F_t(y,z,\mu^Y,\mu^Z,\MZ).
\]
The data $(\xi,\widehat F,G)$ satisfies
Assumption~\ref{assu:H2prime} with parameters
$q>2$, $K>0$, $K_0>0$, and $G_2<\infty$ as specified there. 
Moreover, $h:[0,T] \times \op_1(\R ) \times \op_1(\R^d)
\rightarrow \R$ is deterministic, Borel measurable, 
and satisfies the following quadratic growth condition: for any $\nu_1, \nu_2 \in \op_1(\R),$
$\rho_1, \rho_2 \in \op_1(\R^d),$
\begin{eqnarray*}
&&|h_t(\nu_1, \rho_1)- h_t(\nu_2, \rho_2)| \le K \left[ W_1(\nu_1 , \nu_2)+(1+\|\rho_1\|_{1} + \|\rho_2\|_{1}) W_1(\rho_1,\rho_2)  \right],\\
&& |h_t(\delta_0,\rho)|\le k_t+ K \| \rho \|_1^2,	\ \ \ \text{where }k \in L^1([0,T]).
\end{eqnarray*}

\end{assumption}

\begin{rem}
The quadratic bounds in Assumption~\ref{assu:H3} can be replaced by
subquadratic ones.  More precisely, for some $q'\in[1,2]$ one may assume
\begin{eqnarray*}
&&|h_t(\nu_1, \rho_1)- h_t(\nu_2, \rho_2)| \le K \left[ W_1(\nu_1 , \nu_2)+(1+\|\rho_1\|_{1}^{q'-1} + \|\rho_2\|_{1}^{q'-1}) W_1(\rho_1,\rho_2)  \right],\\
&& |h_t(\delta_0,\rho)|\le k_t+ K \| \rho \|_1^{q'},	\ \ \ \text{where }k \in L^1([0,T]).
\end{eqnarray*}	
In this case the local well-posedness conclusion of
Theorem~\ref{mainlocal} remains valid.  
	
\end{rem}

In this case, instead of building a fixed point argument for \eqref{bsde-law-lb}, we study the following system,
\begin{eqnarray}\label{y1}
	&& \yone_t=\xi + \int_t^T F_r(\yone_r + \ytwo_r, z^1_r, \ML(\yone_r + \ytwo_r), \ML(\zone_r), G(\ML({\zone_r})) \zone_r) dr - \int_t^T \zone_r dB_r, \\ \label{y2}
	&& \ytwo_t= \int_t^T h_r (\ML(\yone_r + \ytwo_r), \ML(\zone_r)) dr.
\end{eqnarray}
Note that although the system \eqref{y1} and \eqref{y2} is coupled, the second equation is a deterministic ODE, which  has trivial Malliavin derivative, while the first equation falls under the scope of Theorem~\ref{local-wp-law} once we fix $y^{(2)}$. This is the key observation for us to solve \eqref{bsde-law-lb}. We have the equivalence on well-posedness between the above system and equation \eqref{bsde-law-lb}.

\begin{prop}\label{equiv}
Suppose that Assumption~\ref{assu:H3} holds. For any $t\in [0,T),$ there exists a unique solution $(y,z) \in \MS^q  \times \MM^q $ on $[t,T]$ to \eqref{bsde-law-lb} if and only if there exists a unique solution $(y^1,z^1,y^2)   \in \MS^q  \times \MM^q \times \MC$ to system \eqref{y1} \eqref{y2}  on $[t,T]$. In this case, we have
\begin{equation*}
 y=y^1+y^2,\ \  z=z^1.
\end{equation*}
	
\end{prop} 

\begin{proof}
Without loss of generality, let $t=0.$
For the necessary part, suppose that $(y,z) \in \MS^q \times \MM^q$ is the unique solution to \eqref{bsde-law-lb}. In view of Assumption~\ref{assu:H3}, we have a.e.
$$
| h_r(\ML(y_r), \ML(z_r))| \le k_r + K \|\ML(z_r)\|_1^2 + K W_1(\ML(y_r), \delta_0) \le k_r +K (\E[|z_r|^2]+ \E[|y_r|]).
$$
Let
\be\label{def-y12}
\ytwo_t:= \int_t^T h_r (\ML(y_r), \ML(z_r)) dr \in \MC,\ \text{and} \  (\yone,\zone):= (y-\ytwo, z) \in \MS^q \times \MM^q.
\ee
It follows by \eqref{decom} and \eqref{def-y12} that $(\yone, \zone, \ytwo)$ is a solution to \eqref{y1} and \eqref{y2}. For the uniqueness of solutions to \eqref{y1} and \eqref{y2}, suppose that $(\ty^{(1)}, \tz^{(1)}, \ty^{(2)})$ is another set of solution. Then $(\ty,\tz):= (\ty^{(1)}+\ty^{(2)}, \tz^{(1)} )$ is a solution to \eqref{bsde-law-lb}. It follows by the uniqueness of solution to \eqref{bsde-law-lb} that $(\ty^{(1)}+\ty^{(2)}, \tz^{(1)})=(\yone+\ytwo, \zone)$. Then we have
$$
y^{(2)}_t= \ty^{(2)}_t= \int_t^T h_r(\ML(y_r), \ML(z_r)) dr,
$$
which also implies $y^{(1)}=\ty^{(1)}.$
The sufficient part follows similarly.

%

\end{proof}

Thanks to the above characterization of solutions of \eqref{bsde-law-lb} via \eqref{y1}, \eqref{y2}, we can solve \eqref{bsde-law-lb} by a fixed point argument for the system.

\begin{thm}\label{mainlocal}

	Suppose that Assumption~\ref{assu:H3} holds. Then there exists $\delta >0$ depending on $(q,K,K_0,G_2,\|\xi\|_{\ML^q},\|k\|_{L^1,[0,T]})$ such that there exists a unique $(Y,Z) \in \MS^q \times \MM^q$ on $[T-\delta,T]$ solving mean-field qBSDE \eqref{bsde-law-lb}.
	
\end{thm}

\begin{proof}
	
In view of Proposition \ref{equiv}, we only need to show that there exists a unique solution $(\yone, \zone, \ytwo) \in \MS^q \times \MM^q \times \MC$ on some $[T-\delta, T]$ to the system \eqref{y1} and \eqref{y2}. We will prove this by a fixed point argument.

\textbf{Step 1: Construction of the fixed point mapping.}
For any $(u,v) \in \MS^q([s,T]) \times \MM^q([s,T], \R^d)$ with $s$ to be determined later, consider the decoupled system
\begin{eqnarray}\label{y1'}
	&& \yone_t=\xi + \int_t^T F_r(\yone_r + \ytwo_r, \zone_r, \ML(\yone_r + \ytwo_r), \ML(\zone_r), G(\ML({\zone_r})) \zone_r) dr - \int_t^T \zone_r dB_r, \\ \label{y2'}
	&& \ytwo_t= \int_t^T h_r (\ML(u_r + \ytwo_r), \ML(v_r)) dr, \ \ \ t\in [s,T].
\end{eqnarray}
Note that by Assumption~\ref{assu:H3}, $g_r(y):=h_r(\ML(u_r+ y), \ML(v_r))$ is $K$-Lipschitz in $y$, and
$$
|g_r(0)| \le K \E[|u_r|] + k_r + K \E[|v_r|^2].
$$
It follows by classical estimate for ODEs that 
\be\label{est-y2'}
\begin{split}
\| \ytwo \|_{[t,T]} & \le  e^{KT} \int_t^T |g_r(0)|dr\\
 & \le  e^{KT} (K (T-t) \| u \|_{\MS^q,[t,T]} + \|k\|_{L^1,[t,T]} + K \|v\|_{\MM^2,[t,T]}^2),
\end{split}
\ee
so \eqref{y2'} is well-posed.  For \eqref{y1'}, define, for
$\mu\in\op_1(\R^{d+1})$,
\[
 \tF_t(y,z,\mu,\MZ)
 :=F_t\bigl(y+\ytwo_t,z,\mu^Y+\ytwo_t,\mu^Z,\MZ\bigr),
\]
where $\mu^Y+\ytwo_t$ denotes the translation of the first marginal
by $\ytwo_t$.  Since $\ytwo$ is deterministic, the shifted data
$(\xi,\tF,G)$ satisfy Assumption~\ref{assu:H2prime}, with
$\|D_u\tF_r\|_{\MC_b}\le\|D_uF_r\|_{\MC_b}$ and
\begin{equation*}
\begin{split}
\E \int_t^T |\tF_r(0)|^q dr	& \le \E \int_t^T (|F_r(0)| + 2K|\ytwo_r|)^q dr\\
& \lesssim_{q,K} \E \int_t^T |F_r(0)|^q dr+(T-t) \left(\|u\|_{\MS^q,[t,T]}T+ \|k\|_{L^1,[t,T]} + \|v\|^2_{\MM^2,[t,T]} \right)^q,
\end{split}
\end{equation*}
where we apply \eqref{est-y2'} in the second inequality above.
It follows that 
\be \label{K0'}
\begin{split}
&\E [\int_s^T |D_{u} \xi |^q du + \int_s^T \int_s^T\|D_u \tF_r(\cdot)\|_{\MC_b}^q du dr + \int_s^T |\tF_r(0)|^q dr ]\\
&\ \ \ \ \ \  \le C_{q,K}\left(K_0 + (T-s)(\|u\|_{\MS^q,[s,T]}(T-s)+ \|k\|_{L^1,[s,T]} + \|v\|^2_{\MM^2,[s,T]})^q \right).
\end{split}
\ee 
Let
\[
 \tK_0:=2C_{q,K}(K_0+1); \qquad 
 \tilde R^q:=e\tK_0\exp(4qK^2G_2).
\]
Now consider 
\be\label{us'}
\begin{split}
\MU_s:= & \Big\{(u,v) \in \MS^q([s,T]) \times \MH^q([s,T], \R^d) :\\
&\quad \|v\|_{\MM^q,[s,T]}^q \le \tilde R^q,\quad
\|u\|_{\MS^q,[s,T]}^q
\le C \bigl(\|\xi\|_{\ML^q}^q+1+\tK_0\bigr)\Big\},
\end{split}
\ee
where $C=C_{q,K, \tilde{K}_0,G_2}>C_{q,K}$ from here on, is a constant depending only on $(q,K,\tK_0, G_2)$, and in particular can be choose as the implicit constant from \eqref{law-local-apriori} with $K_0$ there replaced by $\tK_0.$ Equipped with the $\MS^q\times\MH^q$ metric, by Fatou's lemma, $\MU_s$ is complete.
Choose $s$ close enough to $T$ such that 
$$
(T-s)\left(
T\left[C \bigl(\|\xi\|_{\ML^q}^q+1+\tK_0\bigr)\right]^{1/q}
+\|k\|_{L^1,[0,T]}+\tilde R^2\right)^q<1.
$$
Then \eqref{K0'} implies,
\be \label{K0''}
\E \left[\int_s^T |D_{u} \xi |^q du + \int_s^T \int_s^T\|D_u \tF_r(\cdot)\|_{\MC_b}^q du dr + \int_s^T |\tF_r(0)|^q dr \right]
\le C_{q,K}(K_0+1)<\tK_0.
\ee 
Then choose $s$ such that $T-s$ even smaller, depending only on
$(q,K,K_0,G_2,\|\xi\|_{\ML^q},\|k\|_{L^1,[0,T]})$, such that Theorem~\ref{local-wp-law} holds for equation \eqref{y1'} with $(K_0, R)$ there replaced by $(\widetilde K_0, \widetilde R)$. 
For every $(u,v)\in\MU_s$, by Theorem~\ref{local-wp-law}, there exists 
a unique solution
$(\yone,\zone)\in\MS^q([s,T])\times\MM^q([s,T],\R^d)$ of
\eqref{y1'}.
Consider mapping $\Psi$: $\MU_s \to  \MS^q([s,T])\times\MM^q([s,T],\R^d)$, $$\Psi(u,v):=(\yone, \zone).$$
In the following, we show that $\Psi$ a contraction on $\MU_s$ when
$\ell:=T-s$ is sufficiently small. 

\textbf{Step 2: Invariance and contraction.}
For the invariance part, by \eqref{law-local-apriori} and \eqref{K0''}, we have
\be\label{est:z1}
\|\yone\|_{\MS^q,[s,T]}^q \le C \left(\| \xi \|_{\ML^q}^q + \ell^{q-1} (\tK_0+ \tilde R^q) \right), \qquad \|\zone\|_{\MM^q,[s,T]}^q\le\tilde R^q.
\ee
Choose $\ell $ smaller again such that $ \ell^{q-1} (\tK_0+ \tilde R^q) < 1+ \tK_0$, and we see the invariance of $\Psi.$

For the contraction part, suppose that $(\bu,\bv) \in \MU_s$ and
$(\by^{(1)},\bz^{(1)},\by^{(2)})$ is the solution to
\eqref{y1'}--\eqref{y2'} with $(u,v)$ replaced by $(\bu,\bv)$.
Let
$(\Delta\yone,\Delta\zone,\Delta\ytwo,\Delta u,\Delta v)
:=(\yone-\by^{(1)},\zone-\bz^{(1)},\ytwo-\by^{(2)},u-\bu,v-\bv)$.
It follows by \eqref{y2'} that for any $t\in[s,T]$,
$$
\Delta \ytwo_t = \int_t^T \left[h_r(\ML(u_r)+\ytwo_r, \ML(v_r))- h_r(\ML(\bu_r)+\by^{(2)}_r, \ML(\bv_r)) \right] dr.
$$
Then by Assumption~\ref{assu:H3} we have
\[
\begin{split}
|\Delta \ytwo_t| & \le K \int_t^T \left[ W_1(\ML(u_r)+ \ytwo_r, \ML(\bu_r) + \by^{(2)}_r ) + (1+\|\ML(v_r)\|_1 +\|\ML(\bv_r)\|_1 )W_1(\ML(v_r), \ML(\bv_r)) \right] dr\\
& \le K \left[ (\|\Delta u\|_{\MS^q,[t,T]}+\| \Delta \ytwo \|_{[t,T]} )(T-t) + \int_t^T (1+ \E|v_r| +  \E|\bv_r| ) \E|\Delta v_r| dr\right].
\end{split}
\]
It follows by the above inequality and choosing $T-s<1/(2K)$ that
\be\label{D-y2'}
\begin{split}
\| \Delta \ytwo \|_{[t,T]}
\le C_K\Big[(T-t)\|\Delta u\|_{\MS^q,[t,T]} +\big((T-t)^{1/2}
  +(T-t)^{1/2-1/q}\widetilde R\big)
  \|\Delta v\|_{\MH^q,[t,T]}\Big].
\end{split}
\ee	
On the other hand, a standard telescoping linearization of
\eqref{y1'} gives
\[
 \Delta\yone_t=\int_t^T
 \big(A_r\Delta\yone_r+B_r\Delta\zone_r
      +C_rG_r(\ML(\zone_r))\Delta\zone_r+b_r\big)dr
 -\int_t^T\Delta\zone_r dB_r,
\]
where $A,B,C$ are bounded by a constant depending only on $K$, and
\be\label{est:br}
\begin{split}
 |b_r|\le C_K\big(&|\Delta\ytwo_r|
 +W_1(\ML(\yone_r+\ytwo_r),
       \ML(\by^{(1)}_r+\by^{(2)}_r))\\
 &+W_1(\ML(\zone_r),\ML(\bz^{(1)}_r))
 +W_1(\ML(\zone_r),\ML(\bz^{(1)}_r))
       |\bz^{(1)}_r|\big).
\end{split}
\ee
Moreover,
\be\label{a'}
\int_t^T |G_r(\ML(\zone_r))|^2 dr
\le2G_2+2K^2(T-t)^{1-2/q}\widetilde R^2.
\ee
According to Lemma~\ref{a-bsde} and \eqref{a'}, for $T-t$
sufficiently small,
\be\label{est-delta}
\begin{split}
	\|\Delta \yone \|_{\MS^q,[t, T] }  + \| \Delta  \zone \|_{\MH^q,[t, T]  }     \lesssim_{K,q,G_2,\widetilde R} \|b\|_{\M^q,[t,T]}.
\end{split}
\ee
By \eqref{us'}, \eqref{D-y2'}, and \eqref{est:br}, we have
\be\label{86}
\begin{split}
\|b\|_{\M^q,[t,T]}
& \le C_{K,q}\Big[
 (T-t)\|\Delta\ytwo\|_{[t,T]}
 +(T-t)\|\Delta\yone\|_{\MS^q,[t,T]}\\
&\qquad 
 +(T-t)^{1/2}\|\Delta\zone\|_{\MH^q,[t,T]}
 +(T-t)^{1/2-1/q}\widetilde R
       \|\Delta\zone\|_{\MH^q,[t,T]}\Big],
\end{split}
\ee
where for the product in the last term we use the fact
\[
 \|(\E|\Delta\zone|)|\bz^{(1)}|\|_{\M^q,[t,T]}
 \le\|\Delta\zone\|_{\MH^2,[t,T]}
      \|\bz^{(1)}\|_{\MH^q,[t,T]}
 \le (T-t)^{1/2-1/q}\widetilde R
      \|\Delta\zone\|_{\MH^q,[t,T]}.
\]
Combining \eqref{est-delta} and \eqref{86}, and  chosing $ \ell$ small, we obtain that 
\[
 \|\Delta\yone\|_{\MS^q,[t,T]}
 +\|\Delta\zone\|_{\MH^q,[t,T]}
 \le C_{K,q,G_2,\widetilde R}(T-t)\|\Delta\ytwo\|_{[t,T]}.
\]
It follows by \eqref{D-y2'} and the above inequality that, after decreasing $\ell$ once more,
\[
 \|\Delta\yone\|_{\MS^q,[s,T]}
 +\|\Delta\zone\|_{\MH^q,[s,T]}
 \le\frac12\left(
   \|\Delta u\|_{\MS^q,[s,T]}
   +\|\Delta v\|_{\MH^q,[s,T]}\right).
\]
Hence $\Psi$ has a fixed point
$(\yone,\zone)\in\MS^q\times\MM^q$.  The corresponding deterministic
$\ytwo$ solves \eqref{y2}, and
$Y=\yone+\ytwo\in\MS^q$.

\noindent\textbf{Step 3: Uniqueness in the full space.}
It remains to prove uniqueness in the full space
$\MS^q\times\MM^q\times\MC$, rather than only in the ball $\MU_s$.
Let $(y^{(1)},z^{(1)},y^{(2)})$ and
$(\bar y^{(1)},\bar z^{(1)},\bar y^{(2)})$ be two solutions of
\eqref{y1}--\eqref{y2} on the same interval $[s,T]$, and use the notation $\Delta (y^{(1)},z^{(1)},y^{(2)})$
for their differences. Let
\[
 R_0:=1+\|z^{(1)}\|_{\MM^q,[s,T]}
       +\|\bar z^{(1)}\|_{\MM^q,[s,T]}<\infty,
\]
and
$|I|=T-s$. Repeating
the estimates \eqref{D-y2'}--\eqref{86} on $I$, without using the
a priori bound $\widetilde R$, yields
\[
\begin{split}
 &\|\Delta y^{(1)}\|_{\MS^q,I}
  +\|\Delta z^{(1)}\|_{\MH^q,I}
  +\|\Delta y^{(2)}\|_{I}\\
 &\quad\le C_{K,q,G_2,T,R_0}
 \left(|I|+|I|^{1/2}+|I|^{1/2-1/q}R_0\right)
 \left(\|\Delta y^{(1)}\|_{\MS^q,I}
  +\|\Delta z^{(1)}\|_{\MH^q,I}
  +\|\Delta y^{(2)}\|_{I}\right).
\end{split}
\]
Since $q>2$, we may choose a finite partition
$s=t_0<t_1<\cdots<t_J=T$ with mesh sufficiently small that the
coefficient on the right-hand side is strictly less than one on every
interval $[t_j,t_{j+1}]$. Note that $y_T^{(1)}=\bar y_T^{(1)}=\xi$ and
$y_T^{(2)}=\bar y_T^{(2)}=0$.  The above estimate then shows
that they coincide on $[t_{J-1},T]$.  Iterating backward over the
partition shows that they coincide on all of $[s,T]$. 

\end{proof}


\begin{rem}\label{rem:global-counterexample}
Assumption~\ref{assu:H3} (or Assumption~\ref{assu:H2prime}) is not sufficient for global existence for \eqref{bsde-law-lb}.  Indeed, let $d=1$
and consider the BSDE
\be\label{global-counterexample-bsde}
Y_t=\frac{1}{3}B_T^3
 +\int_t^T \E[Z_r]Z_r\,dr-\int_t^T Z_r\,dB_r.
\ee
In this case, there is explicit solution 
\[
 \begin{split}
 Y_t&=\frac{1}{3}\bigl[(B_t+A_t)^3
       +3(B_t+A_t)(T-t)\bigr],\\
 Z_t&=(B_t+A_t)^2+T-t.
 \end{split}
\]
where $
 A_t=\sqrt T\tan\bigl(\sqrt T(T-t)\bigr).
$ Indeed, suppose there is a global solution. By taking the Malliavin derivative of the above equation, we have 
\[
 D_uY_t=B_T^2+\int_t^T \E[Z_r] D_uZ_rdr
             -\int_t^TD_uZ_rdB_r,
\]
Thus by changing the measure to 
$
\frac{dQ}{dP}:=\mathcal E\left(\int_0^{\cdot} \E[Z_r]  dB_r\right)_T,$ with 
$ B_t^Q:=B_t-\int_0^t \E[Z_r] dr,$ we have
\[
 D_uY_t=\E_t^Q[B_T^2].
\]
Note that $B_T=B_T^Q-B_t^Q + B_t + \int_t^T \E[Z_r] dr$. Taking this to the above identity, we have 
\[
 Z_t=D_tY_t= (B_t + \int_t^T \E[Z_r] dr)^2 + T-t.
\] 
Taking expectation on both sides above, and we obtain the ODE for $A_t:=\int_t^T \E[Z_r] dr$
\[
- A'_t= A_t^2 + T,\qquad A_T=0.
\]
Note that whenever $T^{3/2}\ge\pi/2$, the solution $A$ has a singularity in $[0,T]$. Thus the solution is just local in view of the uniqueness of solutions by Theorem~\ref{mainlocal}.

\end{rem}

\section{Global well-posedness under an exponential integrability condition}
\label{sec:global}
In view of Remark~\ref{rem:global-counterexample}, the local well-posedness result in Theorem~\ref{mainlocal} cannot be extended to a global result without additional assumptions.
Indeed, the polynomial Malliavin bounds in Assumption~\ref{assu:H3} do not provide solution-independent bound of the determinisitc part $G_t(\ML(Z_t))$. The following condition provides the necessary condition needed to bound this coefficient globally via entropy inequality.

\begin{assumption}\label{assu:EI}
Suppose that Assumption~\ref{assu:H3} holds for coefficients $(\xi, H,F,h, G)$ with parameters $(q,K,K_0,G_2)$. For any $u\in[0,T]$, let
\[
 \mathfrak R_u:=|D_u\xi|
   +\int_u^T\|D_uF_r(\cdot)\|_{\MC_b}dr.
\]
Suppose that there is a deterministic Borel function
$\varsigma:[0,T]\to[0,\infty)$ such that
\be\label{EI-condition}
 \Sigma:=\int_0^T\varsigma_u^2du<\infty,
 \qquad
 \E\exp\left(\frac{\mathfrak R_u^2}{\varsigma_u^2}\right)\le2
 \quad\text{for a.e. }u\in[0,T].
\ee
At points where $\varsigma_u=0$, the second condition in
\eqref{EI-condition} means that $\mathfrak R_u=0$ a.s.
\end{assumption}

\begin{rem}[Terminal moment conditions]
\label{rem:EI-terminal-moments}

Although we do not assume more integrability of $\xi$ in Assumption~\ref{assu:EI}, in view of \eqref{EI-condition} and the Gassian Poincar\'e inequality (see \cite[Chapter~VIII Theorem 3]{Ustunel95}, also \cite[Theorem~2.7]{AddonaMuratoriRossi22} ), $|\xi-\E\xi|$ has exponential integrability as well.

\end{rem}


\begin{thm}\label{thm:global-wp}
Suppose that Assumption~\ref{assu:EI} holds.  Then
\eqref{bsde-law-lb} has a unique solution
$(Y,Z)\in\MS^q([0,T])\times\MM^q([0,T],\R^d)$.
Moreover,
\be\label{global-solution-bound}
 \|Y\|_{\MS^q,[0,T]}^q+\|Z\|_{\MM^q,[0,T]}^q
 \le C_{q,K,T,G_2,\Sigma}
 \left(\|\xi\|_{\ML^q}^q+\|k\|_{L^1}^q+K_0+K_0^2\right).
\ee
\end{thm}

\begin{proof}

In view of Theorem \ref{mainlocal}, there exists a unique solution $(Y,Z)\in\MS^q\times\MM^q$ on some interval $[s,T]$. Moreover, since the length $(T-s)$ depends only on the parameters $(q,K,K_0,G_2,\allowbreak\|\xi\|_{\ML^q},\allowbreak\|k\|_{L^1,[0,T]})$, it suffices to prove the a priori estimate of $(\|Y_{s}\|_{\ML^q}, \E \int_0^T |D_uY_{s}|^q\,du )$ independently of $s\in[0,T]$, such that the local solution can be iterated to cover the whole interval $[0,T]$. Now set
\begin{align*}
a_r:=G_r(\ML(Z_r)),
 \quad \Gamma_t:=\int_t^T h_r(\ML(Y_r),\ML(Z_r))dr,\quad g_r(y,z,\MZ):=F_r(y+\Gamma_r,z,\ML(Y_r),\ML(Z_r),\MZ).
\end{align*}
Then we have that for any $t\in [s,T]$,
\be\label{ine:bdd-int-a}
\int_t^T|a_r|^2dr
\le 2G_2 + 2K^2 \int_t^T (\E|Z_r|)^2 dr.
\ee
Thus $(Y-\Gamma,Z)$ solves BSDE \eqref{BSDE2} with corresponding coefficients $(\xi, a,g)$ above.

\textbf{Step 1: Bound on $\E|Z_t|$ via the entropy inequality in the smooth case.}
We first assume that
$F$ is smooth in $(y,z,\MZ)$, i.e. only in its Euclidean arguments. Later in Step~3 we remove this assumption. Note that $D_ug_r=D_uF_r(y+\Gamma_r,z,\ML(Y_r),\ML(Z_r),\MZ)$. According
to Proposition~\ref{m-d}, we have $Z_t=D_tY_t$ and furthermore,
\be\label{eq:global-malliavin}
 \begin{split}
 D_tY_r={}&D_t\xi
   +\int_r^T\bigl[D_tF_v(\Theta_v)
                  +\alpha_vD_tY_v
                  +\beta_vD_tZ_v\bigr]dv
   -\int_r^T D_tZ_vdB_v,
 \qquad r\in[t,T],
 \end{split}
\ee
where $|\alpha_v|\le K$ and
\be\label{eq:global-beta}
 \beta_v=\partial_zF_v+\partial_{\MZ}F_v\,a_v,
 \qquad
 |\beta_v|^2\le2K^2(1+|a_v|^2).
\ee
Here $\Theta_v$ denotes all the arguments of $F_v$. 
Now define
\[
 \frac{dQ^t}{dP}
 :=\mathcal E\left(\int_t^{\cdot}\beta_r dB_r\right)_T.
\]
In view of \eqref{eq:global-beta}, the above
stochastic exponential is a uniformly integrable martingale.  Hence by Girsanov's theorem and \eqref{eq:global-malliavin}, 
\be\label{global-trace-representation}
 \begin{split}
 |Z_t|=|D_tY_t|
 \le e^{KT}\E_t^{Q^t}\left[
       |D_t\xi|+\int_t^T|D_tF_r(\Theta_r)|dr\right] \le e^{KT}\E_t^{Q^t}[\mathfrak R_t].
 \end{split}
\ee
In the following we establish a bound on the right-hand side that is independent of the solution $(Y,Z)$ by the entropy inequality. 
Let $L_T^t:=dQ^t/dP$ and define the relative entropy
\[
 \mathcal H(Q^t\mid P):=\E^{Q^t}\log L_T^t.
\]
Under $Q^t$,
$B_r^{Q^t}:=B_r-\int_t^r\beta_vdv$, $r\in[t,T]$, is a Brownian
motion, and hence
\[
 \log L_T^t
 =\int_t^T\beta_r dB_r^{Q^t}
  +\frac12\int_t^T|\beta_r|^2dr.
\]
Therefore, in view of \eqref{eq:global-beta}, we have
\be\label{global-entropy-bound}
 \mathcal H(Q^t\mid P)
 =\frac12\E^{Q^t}\int_t^T|\beta_r|^2dr
 \le K^2\left(T-t+\int_t^T|a_r|^2dr\right).
\ee
On the other hand, by Young's inequality and
\eqref{EI-condition}, for any $\lambda>0$,
\be\label{eq:bd-log-e}
 \log\E e^{\lambda\mathfrak R_t}
 \le\log2+\frac{\lambda^2\varsigma_t^2}{4}.
\ee
Applying the entropy inequality (see, e.g.,
\cite[Proposition~1.4.2]{DupuisEllis97}), \eqref{global-entropy-bound} and \eqref{eq:bd-log-e}, we have
\be\label{ine-entropy}
 \lambda\E^{Q^t}\mathfrak R_t
 \le \mathcal H(Q^t\mid P)+\log\E e^{\lambda\mathfrak R_t}
 \le \mathcal H(Q^t\mid P)+\log2
       +\frac{\lambda^2\varsigma_t^2}{4}.
\ee
Dividing the above inequality by $\lambda$, we have 
\[
 \E^{Q^t}\mathfrak R_t
 \le\frac{\mathcal H(Q^t\mid P)+\log2}{\lambda}
     +\frac{\lambda\varsigma_t^2}{4}.
\]
Note that the right-hand side is minimized at
$\lambda=2\sqrt{\mathcal H(Q^t\mid P)+\log2}/\varsigma_t$, and hence
\be\label{est:Rt}
 \big(\E^{Q^t}\mathfrak R_t\big)^2
 \le\varsigma_t^2\big(\mathcal H(Q^t\mid P)+\log2\big).
\ee
Note that $Q^t$ and $P$ agree on $\MF_t$ in view of the fact that $\E[L_T^t\mid\MF_t]=1$. Thus by \eqref{global-trace-representation}, we have
\[
 \E|Z_t|
 \le e^{KT}\E\big[\E_t^{Q^t}\mathfrak R_t\big]
 =e^{KT}\E^{Q^t}\mathfrak R_t.
\]
Then by the above inequality, \eqref{est:Rt} and \eqref{global-entropy-bound}, we have
\be\label{est:mean-Z}
\begin{split}
 (\E|Z_t|)^2
 \le e^{2KT}\varsigma_t^2
       \left(\log2+K^2T+K^2\int_t^T|a_r|^2dr\right).
\end{split}
\ee
Let
$K_T:=\log2+K^2T$.  
Thus in view of \eqref{ine:bdd-int-a} and \eqref{est:mean-Z}, we obtain
\begin{equation*}
 \int_t^T|a_r|^2dr\le2G_2+2K^2e^{2KT}
       \int_t^T\varsigma_r^2
       \left(K_T+K^2\int_r^T|a_v|^2dv\right)dr.
\end{equation*}
Then it follows by the backward Gronwall inequality that
\be\label{ine:uni-bd-a}
 \int_t^T|a_r|^2dr\le A_*,
\ee
where 
$
 A_*:=\left(2G_2+2K^2e^{2KT}K_T\Sigma\right)
       \exp\left(2K^4e^{2KT}\Sigma\right).
$

\textbf{Step 2: Global estimates of $Y$ and $DY$.}
Apply Proposition \ref{m-d} to the law frozen equation
\eqref{BSDE2} driven by $(\xi,a,g)$, and we have
\be\label{est:global-Zq}
 \int_s^T\E|Z_t|^qdt
 \le K_0\exp\left\{2q\left[(1+K^2)T+K^2A_*\right]\right\}
 =:M_*.
\ee
Both $A_*$ and $M_*$ are independent of the left endpoint $s$ and $(Y,Z)$.  Moreover,
the above bound implies
\be\label{global-restart-moment-bound}
 \int_s^T(\E|Z_t|)^2dt\le T^{1-2/q}M_*^{2/q}.
\ee
Note that 
$
 |h_t(\ML(Y_t),\ML(Z_t))|
 \le k_t+K\E|Y_t|+K(\E|Z_t|)^2,
$ and it follows that 
\be
\begin{split}
|g_r(0)|
={}&|F_r(\Gamma_r,0,\ML(Y_r),\ML(Z_r),0)|\\
\le{}&K\int_r^T\left(k_t+K\E|Y_t|+K(\E|Z_t|)^2\right)dt
 +K\E|Y_r|+K\E|Z_r|+K | F_r(0)|.
\end{split}
\ee
Then by Lemma~\ref{a-bsde}, applied to the law frozen BSDE, and \eqref{global-restart-moment-bound}, we obtain 
\[
 \E\sup_{t\le r\le T}|Y_r|^q
 \le C_{q,K,T,A_*}\left[
   \|\xi\|_{\ML^q}^q+K_0+\|k\|_{L^1}^q
   +M_*+M_*^2
   +\int_t^T \E\sup_{r\le \ell \le T}|Y_\ell |^q dr\right],
 \qquad t\in[s,T].
\]
Applying the backward Gronwall inequality again and we have
\be\label{global-Y-restart-bound}
 \E \sup_{s\le r\le T} |Y_r|^q
 \le C_{q,K,T,A_*}\left(
 \|\xi\|_{\ML^q}^q+K_0+\|k\|_{L^1}^q+M_*+M_*^2\right)=:C_Y,
\ee
where $C_Y$ is independent of $(s,Y,Z)$.

On the other hand, by applying Proposition~\ref{m-d} to the law frozen BSDE
driven by $(\xi, a,g),$ and noting that $\Gamma$ is deterministic, we have
\be\label{global-restart-derivative}
\begin{split}
 \int_0^s\E|D_uY_s|^qdu
 \le C_*\left[\int_0^s\E|D_u\xi|^qdu
 +\int_0^s\int_s^T
   \E\|D_uF_r(\cdot)\|_{\MC_b}^qdrdu\right]
 \le C_* K_0,
\end{split}
\ee
where
$
 C_*:=C_{q,K,T}
 \exp\left\{2q\left[(1+K^2)T+K^2A_*\right]\right\}.
$
Moreover, by adaptedness of $Y$, we have $D_uY_s=0$ for $u>s$.
In particular, $Y_s\in\ML^q\cap\D^{1,2}$ and moreover, the new coefficients $(Y_s,H, F,h,G)$ 
satisfy Assumption~\ref{assu:H3} on $[0,s]$ with the bound in \eqref{law-assu:K0} replaced by $(1+C_*)K_0$ uniformly in $s$.  

Then we may choose a parameter $\delta_*$, depending only on $(q,K,K_0, C_Y, C_*, T)$, such that the
solution exists in $[(s-\delta_*)\vee0,s]$. Repeating this procedure finitely many times covers $[0,T]$ and proves the global well-posedness. Moreover, note that estimates in Step 1 and Step 2 hold globally, and thus by \eqref{est:global-Zq}, \eqref{global-Y-restart-bound}, we have the global estimate \eqref{global-solution-bound}.

\textbf{Step 3: Removal of the smoothness assumption.}
We now drop the smoothness of $F$.  Let $(Y,Z)$ be a solution on some $[s,T]$, and let $\bar Y:=Y-\Gamma$. Define $(a,\Gamma,g)$ as in Step 1.  Since $q>2$ and $Z\in\MM^q([s,T])$,
\[
 \int_s^T|a_r|^2dr
 \le2G_2+2K^2(T-s)^{1-2/q}\|Z\|_{\MM^q,[s,T]}^2<\infty.
\]
Mollifying only the Euclidean variables $x=(y,z,\MZ)$ and we let 
$g^\varepsilon=g*\phi^\varepsilon$.  Then
\[
 |g^\varepsilon-g|\le CK\varepsilon,\qquad
 \|\nabla g^\varepsilon\|_\infty\le K,\qquad
 \|D_ug_r^\varepsilon(\cdot)\|_{\MC_b}
 \le\|D_uF_r(\cdot)\|_{\MC_b}.
\]
Let $(\bar Y^\varepsilon,Z^\varepsilon)$ solve \eqref{BSDE2} with data
$(\xi,g^\varepsilon,a)$.  By Lemma~\ref{stable1}, we have
\[
 (\bar Y^\varepsilon,Z^\varepsilon)\rightarrow(\bar Y,Z)
 \quad\text{in }\MS^q([s,T])\times\MH^q([s,T]).
\]
Thus along a subsequence,
$Z_t^\varepsilon\to Z_t$ in $L^2(\Omega)$ for a.e. $t$, and hence
$\E|Z_t^\varepsilon|\to\E|Z_t|$.  Therefore \eqref{est:mean-Z} holds for the
original driver. Then by the backward Gronwall argument in Step~2, we now have
\[
 \int_t^T|a_r|^2dr\le A_*,
 \qquad
 \int_s^T\E|Z_t|^qdt\le M_*.
\]
Then \eqref{global-restart-moment-bound} and
\eqref{global-Y-restart-bound} still hold in this case, which provides the uniform estimate of $\|Y\|_{\MS^q,[s,T]}$.

It remains only to show the uniform bound for $D_u Y_s$.
Applying Proposition~\ref{m-d} to the BSDE driven by $(\xi, a, g^{\vep})$, we have
\[
 \int_0^s\E|D_u\bar Y_s^\varepsilon|^qdu
 \le C_*\left[\int_0^s\E|D_u\xi|^qdu
 +\int_0^s\int_s^T
   \E\|D_uF_r(\cdot)\|_{\MC_b}^qdrdu\right]
 \le C_*K_0.
\]
Thus by the weak compactness of $L^2([0,s]\times\Omega)$, along a subsequence,
$D\bar Y_s^\varepsilon\rightharpoonup U$ in $L^2([0,s]\times\Omega)$. Moreover, note that $\bar Y_s^\varepsilon\to\bar Y_s$ in $L^q(\Omega)$. By the weak closedness of the
graph of the Malliavin derivative, we have $\bar Y_s\in\D^{1,2}$ and
$U=D\bar Y_s$. Since $\Gamma_s$ is deterministic, $DY=D \bar Y$, and by the weak lower semicontinuity, we have the bound \eqref{global-restart-derivative} for $D_u Y_s$.

\end{proof}

\begin{example}
Consider the BSDE in Remark~\ref{rem:global-counterexample}, but with terminal
condition $\xi=\frac12B_T^2$, namely
\[
 Y_t=\frac12B_T^2+
 \int_t^T\E[Z_r]Z_r\,dr-\int_t^TZ_r\,dB_r.
\]
This corresponds to
\[
 F_t(y,z,\nu,\rho,w)=w,\qquad
 h_t(\nu,\rho)=0,
 \qquad G_t(\rho)=\int_{\R}z\,\rho(dz),
\]
so the coefficients satisfy Assumption~\ref{assu:H3} for every $q>2$.
Since $F$ is deterministic, for a.e. $u\in[0,T]$ we have
$D_u\xi=B_T$ and hence $\mathfrak R_u=|B_T|$.  Set
\[
 \varsigma_u:=\sqrt{\frac{8T}{3}},\qquad u\in[0,T].
\]
Since $B_T\sim N(0,T)$,
\[
 \int_0^T\varsigma_u^2du=\frac{8T^2}{3},\qquad
 \E\exp\left(\frac{\mathfrak R_u^2}{\varsigma_u^2}\right)
 =\E\exp\left(\frac{3B_T^2}{8T}\right)=2.
\]
Thus Assumption~\ref{assu:EI} holds.  The unique global solution is explicitly
\[
 Y_t=\frac12\bigl(B_t^2+T-t\bigr),\qquad Z_t=B_t,
 \qquad t\in[0,T].
\]
Indeed, $\E[Z_t]=0$, and It\^o's formula gives
\[
 \frac12B_T^2-\int_t^TB_r\,dB_r
 =\frac12\bigl(B_t^2+T-t\bigr).
\]
Moreover, $(Y,Z)\in\MS^q\times\MM^q$ for every finite $q$.
\end{example}

\begin{rem}[A related pathwise quadratic case]
\label{ex:pathwise-quadratic-nonexistence}
The exponential integrability condition \eqref{EI-condition} in
Assumption~\ref{assu:EI} does not ensure global existence if a pathwise
quadratic term is allowed in the generator. Consider the above BSDE but with the generator $g(z)=z^2$, i.e. 
\[
 Y_t=\frac12B_T^2+ \int_t^T Z_r^2\,dr
       -\int_t^T Z_r\,dB_r,\qquad t\in[0,T].
\]
Note that the terminal satisfies Assumption~\ref{assu:EI}.
Nevertheless, the BSDE has no global  solution on $[0,T]$ for any $T>1$.  Indeed, if such a solution existed,
applying It\^o's formula to $e^{2 Y_t}$,  we obtain
\[
 d\bigl(e^{2 Y_t}\bigr)
 =2e^{2 Y_t}Z_t\,dB_t.
\]
Since $Y_0$ is a finite constant under the augmented Brownian filtration,
$e^{2 Y}$ is a positive local martingale and hence a supermartingale.
Consequently,
\[
 \E e^{2 \xi}\le e^{2 Y_0}<\infty.
\]
This contradicts
$\E e^{2 \xi}=\E e^{  B_T^2 }=\infty$ when $  T\ge1$. See Briand and Hu
\cite{BriandHu2006PTRF} for more discussion.

\end{rem}

\section{Applications to the mean-field equilibrium models}
\label{sec:utility-application}

We now apply the results of Sections~\ref{sec:global} to the two models in
Section~\ref{sec:utility-model}.  The
first subsection treats the price-impact portfolio game of
Section~\ref{sec:price-impact-model}, and the second treats the
finite-contract pricing model of Section~\ref{sec:pricing-model}.

\subsection{The mean-field utility maximization game with price-impact}
\label{sec:price-impact-application}

Let $I=[0,T]$. 
Consider the admissible control set
\begin{equation}\label{utility-admissible-class}
 \MA_{I}:=\MM^q(I).
\end{equation}
Recall that for any $\alpha\in\MA_{I}$ and price impact flow $m  $, the state process $X^{\alpha,m}$ is given by
\[
 X_t^{\alpha,m}
 :=x+\int_0^t\alpha_r(b_r+m_r)dr
      +\int_0^t\alpha_r\sigma_r dB_r,
 \qquad t\in I.
\]
The cost functional is given by
\[
 J (\alpha;m)
 :=\E\left[-\exp\left(-\eta
       (X_T^{\alpha,m}-\xi)\right)\right].
\]

\begin{defi}\label{def:local-mf-equilibrium}

A mean-field equilibrium on $I=[0,T]$ is a pair $(\widehat\alpha,m)$ with
deterministic $m\in L^2(I)$ and
$\widehat\alpha\in\MA_{I}$ such that
\be\label{eq:local-mf-equilibrium}
 J (\widehat\alpha;m)
 =\sup_{\alpha\in\MA_{I}}J (\alpha;m),
 \text{ and } \ \ m_t=\E[\widehat\alpha_t]\quad\text{for a.e. }t\in I.
\ee

\end{defi}

For the well-posedness of the mean-field equilibrium, we assume that 

\begin{assumption}
\label{assu:price-impact-equilibrium}
Fix $q>2$ and suppose that $\eta>0$.

(i). Suppose $b$, $\sigma$ are deterministic, and for some $L> 0$ and
$\varepsilon>0$, 
$$ |b_t|+|\sigma_t|+|\sigma_t^{-1}|\le L,
 \ 
 c_t^{-1}:= 1-\frac{1}{\eta \sigma_t^2} \ge\varepsilon. $$ 
(ii). $\xi \in \ML^q \cap \D^{1,2}$ satisfies $\int_0^T\E|D_t\xi|^qdt<\infty$, and moreover, there is a
deterministic Borel function $\varsigma:[0,T]\to[0,\infty)$ such that
\begin{equation}\label{utility-exponential-data}
 \int_0^T\varsigma_u^2du<\infty,
 \qquad
 \E\exp\left(\frac{|D_u\xi|^2}{\varsigma_u^2}\right)\le2
 \quad\text{for a.e. }u\in[0,T].
\end{equation}

\end{assumption}

\begin{rem}

Note that for any fixed impact flow $m\in L^2(I)$, set
\[
 \theta_t^m:=\frac{b_t+m_t}{\sigma_t},\qquad
 \Lambda_t^{m}:=
 \mathcal E\left(-\int_0^\cdot\theta_r^m\,dB_r\right)_t,
 \qquad \frac{dQ^{m}}{dP}:=\Lambda_T^{m}.
\]
Under Assumption \ref{assu:price-impact-equilibrium}, we have
\[
 \int_0^T|\theta_t^m|^2dt
 \le C\bigl(|I|+\|m\|_{L^2(I)}^2\bigr),
\]
so $\Lambda^{m}$ is a uniformly integrable martingale with moments of
every order.  Thus $Q^{m}$ is well defined, and
$B_t^{m}:=B_t+\int_0^t\theta_r^m dr$, $t\in I$, is a Brownian
motion under $Q^{m}$.  For a predictable control $\alpha$, define
\[
 X_t^{\alpha,m}
 :=x+\int_0^t\alpha_r(b_r+m_r)dr
      +\int_0^t\alpha_r\sigma_r dB_r,
 \qquad t\in I,
\]
so that $X_t^{\alpha,m}=x+\int_0^t\alpha_r\sigma_r\,dB_r^{m}$ under
$Q^{m}$.  Since $q>2$ and $\Lambda_T^{m}$ has moments of every order. Moreover, by H\"older's inequality, we see that for any $\alpha\in\MA_{I}$
\[
 \E^{Q^{m}}\int_0^T|\alpha_t\sigma_t|^2dt<\infty.
\]
Thus our admissible strategy class
$\MA_{I}=\MM^q(I)$ is consistent with the usual admissibility
condition (see \cite[Definition 1]{HuImkellerMuller05} and remarks therein).  

\end{rem}
\begin{rem}

Indeed, we may apply results of Section \ref{sec:local-solvability} to  consider local equilibria that allows the drift to be stochastic. However, to directly apply Theorem \ref{thm:global-wp}, we assume deterministic coefficients in the model so the exponential integrability condition \eqref{EI-condition} naturally applies.

\end{rem}

\begin{thm}\label{thm:utility-equilibrium}
Suppose that Assumption~\ref{assu:price-impact-equilibrium} holds. Then \eqref{utility-mf-qbsde} has a unique solution
$(Y,Z)\in\MS^q([0,T])\times\MM^q([0,T])$. Moreover, define
\begin{align}\label{eq:utility-equilibrium-alpha-m}
 m_t^* :=c_t\left(\frac{\E[Z_t]}{\sigma_t}
       +\frac{b_t}{\eta\sigma_t^2}\right),
 \quad
 \alpha_t^* :=\frac{Z_t}{\sigma_t}
   +\frac{c_t}{\eta\sigma_t^2}
       \left(\frac{\E[Z_t]}{\sigma_t}+b_t\right)
 =\frac{Z_t}{\sigma_t}
   +\frac{b_t+m_t^*}{\eta\sigma_t^2}.
\end{align}
Then $m^*\in L^q([0,T])$,
$\alpha^*\in\MA$, and
$(\alpha^*,m^*)$ is the unique mean-field equilibrium in the sense of Definition~\ref{def:local-mf-equilibrium}.
\end{thm}

\begin{proof}
\textbf{Step 1: Verification of the BSDE assumptions.}
Recall from Section~\ref{sec:utility-model} that
$G_t(\rho):=\bar z(\rho):=\int_{\R}z\,\rho(dz)$.  Note that 
$1+c_t/(\eta\sigma_t^2)=c_t$, and thus  the driver $H=F+h$ has representation
\[
\begin{split}
 &F_t(y,z,\nu,\rho,\MZ)=-\frac{c_t}{\sigma_t^2}\MZ
   -\frac{c_tb_t}{\sigma_t}z
   -\frac{c_t^2b_t}{\eta\sigma_t^3}\bar z(\rho)
   -\frac{c_t^2 b_t^2}{2\eta\sigma_t^2},\\
 &h_t(\nu,\rho)
 =-\frac{c_t^2}{2\eta\sigma_t^4}|\bar z(\rho)|^2.
\end{split}
\]
By Assumption~\ref{assu:price-impact-equilibrium}, $F$ is deterministic and globally Lipschitz, with
$F_t(0)=-c_t^2b_t^2/(2\eta\sigma_t^2)$ bounded. Moreover, note that 
\[
 |\bar z(\rho)-\bar z(\rho')|\le W_1(\rho,\rho'),
 \qquad G_t(\delta_0)=0,
\]
and $h$ satisfies quadratic growth
conditions in Assumption~\ref{assu:H3}.  
Hence $(\xi,H, F,h,G)$ satisfies that
 Assumption~\ref{assu:H3}. Since $D_uF=0$, $\mathfrak R_u=|D_u\xi|$, so Assumption~\ref{assu:EI} holds. Then Theorem~\ref{thm:global-wp} yields the unique global solution
$(Y,Z)\in\MS^q\times\MM^q$.

\textbf{Step 2: Optimality and consistency.}
By Jensen's inequality and the boundedness of the coefficients, we have
\[
 \|m^*\|_{L^q}^q+\|\alpha^*\|_{\MM^q}^q
 \le C\left(1+\int_0^T\E|Z_t|^qdt\right)<\infty.
\]
Moreover, by the fact $c_t^{-1}=1-(\eta\sigma_t^2)^{-1}$, taking the expectation on both sides of \eqref{eq:utility-equilibrium-alpha-m}, we obtain the consistency condition
\[
 \E[\alpha_t^*]
 =\frac{\E[Z_t]}{\sigma_t}
  +\frac{b_t+m_t^*}{\eta\sigma_t^2}=m_t^*,
 \quad \text{and moreover, } \quad
 b_t+m_t^*=c_t\left(\frac{\E[Z_t]}{\sigma_t}+b_t\right).
\]
Thus the driver of \eqref{utility-mf-qbsde} can be written as
\[
 f_t^{m^*}(Z_t)
 =-\frac{b_t+m_t^*}{\sigma_t}Z_t
  -\frac{|b_t+m_t^*|^2}{2\eta\sigma_t^2}.
\]
Let
\[
 \Lambda_t^*:=\mathcal E\left(-\int_0^\cdot
       \frac{b_r+m_r^*}{\sigma_r}\,dB_r\right)_t,
 \qquad \frac{dQ^*}{dP}:=\Lambda_T^*,
 \qquad B_t^*:=B_t+\int_0^t\frac{b_r+m_r^*}{\sigma_r}\,dr.
\]
Thus $\Lambda^*$ is
a martingale with moments of every order and $B^*$ is a $Q^*$-Brownian
motion. For any $\alpha\in\MA=\MM^q$, by Girsanov's theorem, we have
\be\label{eq:X-qmartingale}
 X_t^{\alpha,m^*}=x+\int_0^t\alpha_r\sigma_r\,dB_r^*.
\ee
Moreover, by H\"older's inequality,
\[
\begin{split}
 \E^{Q^*}\int_0^T|\alpha_t\sigma_t|^2dt
 &\le \|\sigma\|_\infty^2
 \left(\E[(\Lambda_T^*)^{q/(q-2)}]\right)^{(q-2)/q}
 \left(T^{q/2-1}\E\int_0^T|\alpha_t|^qdt\right)^{2/q}<\infty.
\end{split}
\]
Thus $X^{\alpha,m^*}$ is a true $Q^*$-martingale and
$\E^{Q^*}[X_T^{\alpha,m^*}]=x$ for every $\alpha\in\MA$. When
$\alpha=\alpha^*$, the wealth process and the associated BSDE are
\[
 dX_t^{\alpha^*,m^*}
 =\left(Z_t+\frac{b_t+m_t^*}{\eta\sigma_t}\right)dB_t^*,
 \qquad
 dY_t=\frac{|b_t+m_t^*|^2}{2\eta\sigma_t^2}dt+Z_t\,dB_t^*,
\]
and thus
\be\label{eq:value-identity}
 e^{-\eta(X_T^{\alpha^*,m^*}-\xi)}
 =e^{-\eta(x-Y_0)}\Lambda_T^*.
\ee
On the other hand, for any $\alpha\in\MA$, by the concavity of $U$ and the above identity, we have
\be\label{ineq:concavity-gap}
 U(X_T^{\alpha,m^*}-\xi)
 \le U(X_T^{\alpha^*,m^*}-\xi)
 +\eta e^{-\eta(x-Y_0)}\Lambda_T^*
   (X_T^{\alpha,m^*}-X_T^{\alpha^*,m^*}).
\ee
By \eqref{eq:X-qmartingale}, we have
\[
 \E\!\left[\Lambda_T^*
   (X_T^{\alpha,m^*}-X_T^{\alpha^*,m^*})\right]
 =\E^{Q^*}[X_T^{\alpha,m^*}-X_T^{\alpha^*,m^*}]=0.
\]
Then by taking expectations on both sides of \eqref{ineq:concavity-gap} and using \eqref{eq:value-identity}, we obtain
\[
 J(\alpha;m^*)\le J(\alpha^*;m^*)
 =-e^{-\eta(x-Y_0)},
\]
which implies that $(\alpha^*,m^*)$ is an equilibrium.

Now we show that, for the fixed flow $m^*$, $\alpha^*$ is the
unique optimizer.  Suppose that $\alpha\in\MA$ such that
$J(\alpha;m^*) = J(\alpha^*;m^*)$, and let
$
 \Delta X :=X_T^{\alpha,m^*}-X_T^{\alpha^*,m^*}.
$
In view of \eqref{ineq:concavity-gap}, we have
\[
 \begin{split}
 D:= &U(X_T^{\alpha^*,m^*}-\xi)
       +U'(X_T^{\alpha^*,m^*}-\xi) \Delta X
       -U(X_T^{\alpha,m^*}-\xi)\\
  = &e^{-\eta(X_T^{\alpha^*,m^*}-\xi)}
    \bigl(e^{-\eta\Delta X}-1+\eta\Delta X\bigr) \ge 0,
 \end{split}
\]
On the other hand, by the assumption on $\alpha,$ $\E D =0$ and thus $D=0$ a.s. Then by strict concavity of $U$, $\Delta X=0 \text{ a.s.}$, which implies
$
 X_T^{\alpha,m^*}=X_T^{\alpha^*,m^*}, \  P\text{-a.s.}
$
Since $Q^*$ is equivalent to $P$, the above identity holds
$Q^*$-a.s. Then by It\^o's isometry, we have
\[
 0=\E^{Q^*}\!\left[
   \left|X_T^{\alpha,m^*}-X_T^{\alpha^*,m^*}\right|^2\right]
 =\E^{Q^*}\!\int_0^T
   |(\alpha_t-\alpha_t^*)\sigma_t|^2dt.
\]
It follows that $(\alpha-\alpha^*)\sigma=0$
$dt\otimes dQ^*$-a.e. and thus
$\alpha=\alpha^*$ $dt\otimes dP$-a.e.

\textbf{Step 3: Uniqueness of the equilibrium.}
Let $(\bar\alpha,\bar m)$ be another equilibrium in the stated class and
set
\[
 A_t:=\int_0^t\frac{|b_r+\bar m_r|^2}{2\eta\sigma_r^2}dr.
\]
Since $(b+\bar m)/\sigma\in L^q $, by Lemma~\ref{a-bsde} with $a=(b+\bar m)/\sigma$, $g_r(y,z,\MZ)=-\MZ$, and terminal condition
$\xi-A_T$, we have that there exists a unique
$(\bar Y,\bar Z)\in\MS^q\times\MH^q$ solving
\be\label{eq:barBSDE}
 \bar Y_t=\xi+\int_t^T\left[
 -\frac{b_r+\bar m_r}{\sigma_r}\bar Z_r
 -\frac{|b_r+\bar m_r|^2}{2\eta\sigma_r^2}
 \right]dr-\int_t^T\bar Z_r\,dB_r.
\ee
Moreover, by Proposition~\ref{m-d}, we have $\bar Z \in \MM^q$.
Then let
\[
 \alpha_t^{\bar m}:=\frac{\bar Z_t}{\sigma_t}
 +\frac{b_t+\bar m_t}{\eta\sigma_t^2}.
\]
Then by Step~2, for the given deterministic flow $\bar m$, $\alpha^{\bar m}\in\MA$ is an optimizer.  On the other hand, since $(\bar\alpha,\bar m)$ is an
equilibrium, by the uniqueness of the optimizer, we have
$
 \bar\alpha=\alpha^{\bar m},\  dt\otimes dP\text{-a.e.}
$
Then it follows by the consistency condition that 
\[
\begin{split}
 \bar m_t
 =\E[\bar\alpha_t]
 =\frac{\E[\bar Z_t]}{\sigma_t}
   +\frac{b_t+\bar m_t}{\eta\sigma_t^2} =c_t\left(\frac{\E[\bar Z_t]}{\sigma_t}
             +\frac{b_t}{\eta\sigma_t^2}\right),
\end{split}
\]
where the last equality follows from the fact that $c_t^{-1}=1-(\eta\sigma_t^2)^{-1}$.  
Substituting the expression of $\bar m$ into \eqref{eq:barBSDE}, we obtain
\[
\begin{split}
 \bar Y_t=\xi+\int_t^T\bigg[
 -\frac{c_r}{\sigma_r}
   \left(\frac{\E[\bar Z_r]}{\sigma_r}+b_r\right)\bar Z_r -\frac{c_r^2}{2\eta\sigma_r^2}
   \left|\frac{\E[\bar Z_r]}{\sigma_r}+b_r\right|^2
 \bigg]dr-\int_t^T\bar Z_r\,dB_r,
\end{split}
\]
which is precisely equation \eqref{utility-mf-qbsde}. Then by the uniqueness of the global mean-field BSDE from Step~1, we have $(\bar Y,\bar Z)=(Y,Z)$, and thus $\bar m=m^*$ and
$\bar\alpha=\alpha^*$ $dt\otimes dP$-a.e.
\end{proof}


\subsection{The mean-field equilibrium of finite-contract pricing model }
\label{sec:pricing-application}

Now we consider the finite-contract market of Section~\ref{sec:pricing-model}. Let
$\nu\in L^q([0,T],\R^d)$ be a deterministic return rate.
Given a portfolio $\alpha$ of signed contract positions, the representative client has wealth process
\begin{equation*}
 X_t^{\alpha,\nu}
 =x+\int_0^t\alpha_r^\top\nu_rdr
      +\int_0^t\alpha_r^\top\Sigma_r dB_r.
\end{equation*}
Given $\nu$, the representative client maximizes
\begin{equation*}
 J(\alpha;\nu):=
 \E\left[-\exp\left(-\eta(X_T^{\alpha,\nu}-\xi)\right)\right],
\end{equation*}
over the admissible strategy class
\begin{equation*}
 \MA(\nu):= \MA:= \MM^q([0,T],\R^d).
\end{equation*}
The intermediary provides a deterministic signed supply $s:[0,T]\to\R^d$. The equilibrium satisfies the consistency condition 
$\E[\alpha_t]=s_t$. 

For the well-posedness of the mean-field BSDE \eqref{pricing-mf-qbsde}, we assume

\begin{assumption}
\label{assu:pricing-equilibrium}
(i). Fix $q>2$ and suppose that $\eta>0$.  Let
$\Sigma:[0,T]\to\R^{d\times d}$ be deterministic and invertible, with
$\Sigma$ and $\Sigma^{-1}$ bounded, and let
$s:[0,T]\to\R^d$ be bounded and measurable.\\
(ii). Assume that
$
 \xi\in\ML^q\cap\D^{1,2},
 \ \E\int_0^T|D_u\xi|^qdu<\infty.
$
Moreover, suppose that there is a deterministic Borel function
$\varsigma:[0,T]\to[0,\infty)$ such that
\be\label{ineq:pricing-malliavin-bound}
 \int_0^T\varsigma_u^2du<\infty,
 \qquad
 \E\exp\left(\frac{|D_u\xi|^2}{\varsigma_u^2}\right)\le2
 \quad\text{for a.e. }u\in[0,T].
\ee
\end{assumption}

\begin{defi}[Mean-field pricing equilibrium]
\label{def:pricing-mf-equilibrium}
A mean-field pricing equilibrium is a pair $( \alpha^*,\nu^*)$ with
deterministic $\nu^*\in L^q([0,T],\R^d)$ and
$ \alpha^* \in\MA(\nu^*)$ such that
\[
 J(\alpha^*;\nu^*)
 =\sup_{\alpha\in\MA(\nu^*)}J(\alpha;\nu^*),
\]
and optimal demand is consistent with the prescribed supply:
\[
 \E[\alpha^*_t]=s_t
 \quad\text{for a.e. }t\in[0,T].
\]
\end{defi}

\begin{thm}\label{thm:pricing-equilibrium}
Suppose that Assumption~\ref{assu:pricing-equilibrium} holds.  Then
\eqref{pricing-mf-qbsde} has a unique solution
$(Y,Z)\in\MS^q([0,T])\times\MM^q([0,T],\R^d)$. Moreover, let 
\begin{equation*}
 \nu_t^*:=\eta\Sigma_t(\Sigma_t^\top s_t-\E[Z_t]),
 \qquad
 \alpha_t^*:=s_t+(\Sigma_t^\top)^{-1}
       (Z_t-\E[Z_t]).
\end{equation*}
Then $\nu^*\in L^q([0,T],\R^d)$,
$\alpha^*\in\MA(\nu^*)$, and
$(\alpha^*,\nu^*)$ is the unique mean-field pricing equilibrium in the
sense of Definition~\ref{def:pricing-mf-equilibrium}.
\end{thm}

\begin{proof}
For $\rho\in\op_1(\R^d)$, let
$\bar z(\rho):=\int_{\R^d}z\rho(dz)$, and let 
\[
 G_t(\rho)=\bar z(\rho)-\Sigma_t^\top s_t,\qquad
 F_t(y,z,\nu,\rho,\MZ)=\eta \MZ,
 \qquad h_t(\nu,\rho)
 =-\frac{\eta}{2}|\bar z(\rho)-\Sigma_t^\top s_t|^2.
\]
Then it is easy to see that these coefficients satisfy
Assumption~\ref{assu:EI}, and thus by Theorem~\ref{thm:global-wp}, there exists a unique solution $(Y,Z)$ to \eqref{pricing-mf-qbsde} in $\MS^q\times\MM^q$.

Now we check consistency and optimality. Taking expectations on both sides of the definition of $\alpha^*$ and we obtain the consistency condition. For the optimality, we first show that $\alpha^*$ is the unique optimal control given the pricing flow $\nu^*$, which implies $(\alpha^*, \nu^*)$ is a mean-field pricing equilibrium in the sense of Definition~\ref{def:pricing-mf-equilibrium}. Since $Z\in\MM^q$, by Jensen's
inequality, we see that  $t\mapsto\E[Z_t]\in L^q([0,T],\R^d)$. Hence
$\nu^*\in L^q$ and
$\alpha^*\in\MM^q([0,T],\R^d)$. Let  
$$
 \theta_t^*:=\Sigma_t^{-1}\nu_t^*
 =\eta(\Sigma_t^\top s_t-\E[Z_t]), \qquad \frac{dQ^*}{dP}:= \ME\left(-\int_0^\cdot \theta_r^* dB_r\right)_T, \qquad B_t^*:=B_t+\int_0^t\theta_r^*dr.
$$
Then the state process $X^{\alpha^*,\nu^*}$ satisfies
\[
 dX_t^{\alpha^*,\nu^*}
   =(Z_t-\E[Z_t]+\Sigma_t^\top s_t)^\top dB_t^*.
\]  
Then following a similar martingale optimality argument as in the proof of Theorem~\ref{thm:utility-equilibrium}, we see that $\alpha^*$ is the unique optimal control given $\nu^*$. 



For the uniqueness of equilibria, let $(\bar\alpha,\bar\nu)$ be another equilibrium in the sense of
Definition~\ref{def:pricing-mf-equilibrium} and set
$\bar\theta=\Sigma^{-1}\bar\nu$. Note that $|\bar\theta|^2 \in \BM^q$. Then by Lemma \ref{a-bsde} and Proposition \ref{m-d}, the following BSDE
\be\label{eq:bar-pricing-bsde}
 \bar Y_t=\xi+\int_t^T\left[
 -\bar\theta_r\bar Z_r-\frac{|\bar\theta_r|^2}{2\eta}
 \right]dr-\int_t^T\bar Z_r dB_r
\ee
has a unique solution in $\MS^q\times\MM^q$.  Its unique
optimal portfolio is
\[
 \bar\alpha_t=(\Sigma_t^\top)^{-1}
       \left(\bar Z_t+\frac{\bar\theta_t}{\eta}\right).
\]
By the consistency condition, we have
\[
 \bar\theta_t=\eta(\Sigma_t^\top s_t-\E[\bar Z_t]).
\]  
Substitution the above into \eqref{eq:bar-pricing-bsde} and we have that
$(\bar Y,\bar Z)$ solves \eqref{pricing-mf-qbsde}.  By uniqueness of
solutions to \eqref{pricing-mf-qbsde} we have
$(\bar Y,\bar Z)=(Y,Z)$, and thus
$\bar\nu=\nu^*$ and $\bar\alpha= \alpha^*$.
\end{proof}


\begin{rem}[Convergence of the multiple-player pricing model to the mean-field limit]
\label{rem:pricing-finite-convergence}
Fix the parameters $T,d,q,\eta,\Sigma$, the fixed supply $s$, and a
generic pair $(B,\xi)$ as in
Assumption~\ref{assu:pricing-equilibrium}. Suppose that on a probability space $(\Omega,\mathcal{F},\mathbb{P})$, $(B^i,\xi^i)_{i\ge1}$ are independent copies of $(B,\xi)$.  In
particular, $\xi^i$ is measurable with respect to the augmented filtration
generated by $B^i$.

For every $N$, suppose that an  equilibrium
$(\nu^N,\widehat{\boldsymbol\alpha}^N)$ in the sense of
Definition~\ref{def:finite-client-pricing-equilibrium} exists and is represented by the  multidimensional qBSDE 
\eqref{finite-client-fixed-control}--\eqref{finite-client-equilibrium-bsde}, which is assumed to be well-defined. More precisely, for each $i\le N$, suppose that there are
\[
 Y^{i,N}\in\MS^q([0,T]),\qquad
 (Z^{i,j,N})_{j=1}^N
 \in\MH^q([0,T],\R^{Nd}),
\]
such that $(Y^{i,N},(Z^{i,j,N})_{j=1}^N)_{i=1}^N$ solves
\eqref{finite-client-equilibrium-bsde}, the feedback controls in
\eqref{finite-client-fixed-control} are admissible, and the corresponding
martingale optimality verification is valid.  
Let
\[
 \overline Z_t^N:=\frac1N\sum_{i=1}^N Z_t^{i,i,N}.
\]
For each $i\ge1$, let $(Y^i,Z^i)$ be the
solution to the mean-field BSDE driven by $(B^i,\xi^i)$ and same deterministic
coefficients, that is,
\[
\begin{split}
 Y_t^i= \xi^i+\int_t^T\left[
 \eta\bigl(\E[Z_r]-\Sigma_r^\top s_r\bigr)^\top Z_r^i
 -\frac{\eta}{2}
  \bigl|\E[Z_r]-\Sigma_r^\top s_r\bigr|^2
 \right]dr
 -\int_t^T (Z_r^i)^\top dB_r^i.
\end{split}
\]
Then $(B^i,\xi^i,Y^i,Z^i)_{i\ge1}$ are independent copies of
$(B,\xi,Y,Z)$. Let
\[
 r_N:=\frac1N\sum_{i=1}^N
 \E\int_0^T|Z_t^{i,i,N}-Z_t^i|^2dt,
\]
and assume that $r_N\to0$.

Consider the client return rates in the multiple-player and mean-field equilibria:
\[
 \nu_t^N=\eta\Sigma_t(\Sigma_t^\top s_t-\overline Z_t^N),
 \qquad
 \nu_t^*=\eta\Sigma_t(\Sigma_t^\top s_t-\E[Z_t]).
\]
To show convergence of these return rates, let
\be\label{eq:pricing-convergence}
 \overline Z_t^N-\E[Z_t]
 =\frac1N\sum_{i=1}^N(Z_t^{i,i,N}-Z_t^i)
  +\frac1N\sum_{i=1}^N(Z_t^i-\E[Z_t]).
\ee
Note that the first term on the right-hand side can be controlled by $r_N$. For the second term, 
by Jensen's inequality and independence, we have
\[
 \E\int_0^T\left|
 \frac1N\sum_{i=1}^N(Z_t^i-\E[Z_t])
 \right|^2dt
 =\frac1N\E\int_0^T|Z_t-\E[Z_t]|^2dt.
\]
Then in view of \eqref{eq:pricing-convergence}, we have 
\[
 \E\int_0^T|\nu_t^N-\nu_t^*|^2dt
 \le 2\eta^2\|\Sigma\|_\infty^2
 \left(r_N+\frac1N\E\int_0^T|Z_t-\E[Z_t]|^2dt\right)
 \longrightarrow0.
\]
In this case, define the independent mean-field controls by
\[
 \alpha_t^{*,i}:=s_t+(\Sigma_t^\top)^{-1}
 (Z_t^i-\E[Z_t]).
\]
and the corresponding multiple-player strategy by 
\[
 \widehat\alpha_t^{i,N}
 =s_t+(\Sigma_t^\top)^{-1}
   (Z_t^{i,i,N}-\overline Z_t^N).
\]
Note that
\[
\begin{split}
 &\frac1N\sum_{i=1}^N\left|
 (Z_t^{i,i,N}-Z_t^i)-(\overline Z_t^N-\E[Z_t])
 \right|^2\\
 &\quad=\frac1N\sum_{i=1}^N|Z_t^{i,i,N}-Z_t^i|^2
 -\Big|\frac1N\sum_{i=1}^N(Z_t^{i,i,N}-Z_t^i)\Big|^2
 +\Big|\frac1N\sum_{i=1}^N(Z_t^i-\E[Z_t])\Big|^2.
\end{split}
\]
Thus by independence and the above representation, we have 
\[
 \frac1N\sum_{i=1}^N\E\int_0^T
 |\widehat\alpha_t^{i,N}-\alpha_t^{*,i}|^2dt
 \le \|\Sigma^{-1}\|_\infty^2
 \left(r_N+\frac1N\E\int_0^T
 |Z_t-\E[Z_t]|^2dt\right)
 \longrightarrow0.
\] 
Note that the existence and verification of the multiple-player equilibria, as well as
the convergence $r_N\to0$, are additional assumptions here and do not follow directly 
from Assumption~\ref{assu:pricing-equilibrium}. Related dimension-free
well-posedness and backward propagation-of-chaos estimates are established in
\cite{HorstSchmidekZhang26} for weakly interacting systems under additional explicit smallness assumptions.



\end{rem}

\section{Conclusion}
\label{sec:conclusion}

We establish local existence and uniqueness for a class of mean-field
quadratic BSDEs with interactions of the forms $\E[Z_t]^\top Z_t$ and
$|\E[Z_t]|^2$, under the stated $L^q$ and Malliavin regularity assumptions,
$q>2$. Freezing the law yields deterministic coefficients in the cross term,
while a BSDE--ODE decomposition handles the additional quadratic law term.
Under an exponential integrability condition on the Malliavin derivatives, entropy estimates extend the local solution to any
finite horizon. The blow-up example shows why polynomial assumptions alone
do not suffice for global solvability.

Under the respective model assumptions, these results yield unique mean-field
equilibria for an exponential-utility portfolio game with price impact and
an intermediated market for client-specific risk-transfer contracts with
exogenous supply. The BSDE solutions determine optimal positions and the
equilibrium impact or premium. 



\bibliography{BSDEs.bib}

\end{document}